\documentclass[10pt]{article}

\usepackage[utf8]{inputenc}
\usepackage[english]{babel}
\usepackage[dvipsnames]{xcolor}
\usepackage{mleftright}
\usepackage{tensor}
\usepackage[nottoc]{tocbibind}
\usepackage{amsmath}
\usepackage{cases}
\usepackage{amsfonts}
\usepackage{amssymb}
\usepackage{array}
\usepackage{mathtools}
\usepackage{tikz}
\usepackage{tikz-cd}
\usepackage{ragged2e}
\usepackage{amsthm}
\usepackage{tabularx}
\usepackage{centernot}
\usepackage{faktor}
\usepackage{rotating}
\usepackage{commath}
\usepackage{comment}
\usepackage{mathabx}
\usepackage{fge}
\usepackage[parfill]{parskip}
\usepackage[shortlabels]{enumitem}
\usepackage{xifthen}
\usepackage{etoolbox}
\usepackage{verbatim}
\usepackage{tikzducks}
\usepackage{ytableau}
\usepackage{subfig}
\usepackage{slashed}
\usepackage{xfrac}
\usepackage[title]{appendix}
\usepackage{multirow}
\usepackage{textcomp}
\usepackage[hyphens]{url}
\usepackage[medium]{titlesec}
\usepackage{empheq}
\usepackage[margin=1in]{geometry}
\usepackage{pdflscape}
\usepackage{afterpage}
\usepackage{capt-of}
\usepackage{lipsum}
\usepackage{hyperref}
\usepackage[font=small,labelfont=bf]{caption}
\usepackage{relsize}
\usepackage{float}

\hypersetup{
    colorlinks=true,
    linkcolor=BrickRed,
    filecolor=Apricot,      
    urlcolor=Rhodamine,
    citecolor=ForestGreen,
    anchorcolor=Fuchsia,
}

\DeclareFontFamily{U}{MnSymbolC}{}
\DeclareSymbolFont{MnSyC}{U}{MnSymbolC}{m}{n}
\DeclareFontShape{U}{MnSymbolC}{m}{n}{
    <-6>  MnSymbolC5
   <6-7>  MnSymbolC6
   <7-8>  MnSymbolC7
   <8-9>  MnSymbolC8
   <9-10> MnSymbolC9
  <10-12> MnSymbolC10
  <12->   MnSymbolC12}{}
\DeclareMathSymbol{\hook}{\mathbin}{MnSyC}{'270}

\newcommand{\bb}[1]{
    \mathbb{#1}
}
\renewcommand{\cal}[1]{
    \mathcal{#1}
}

\newcommand{\sserif}[1]{
    \text{\normalfont{\fontfamily{lmss}\selectfont{#1}}}
}

\newcommand{\rarr}{
    \rightarrow
}
\newcommand{\lrarr}{
    \leftrightarrow
}

\renewcommand{\:}{
    \colon
}

\newcommand{\inner}[2]{
    \langle #1,#2 \rangle
}

\newcommand{\dinner}[2]{
    \langle\!\mkern-1mu\langle #1,#2 \rangle\!\mkern-1mu\rangle
}
\newcommand{\dbinner}[2]{
    \big\langle\!\!\big\langle #1,#2 \big\rangle\!\!\big\rangle
}
\renewcommand{\del}[1]{
    \partial #1
}
\newcommand{\delbar}[1]{
    \overline{\partial} #1
}
\newcommand{\tr}[1]{
    \text{\normalfont tr}\,#1
}

\newcommand{\wtilde}[1]{
    \widetilde{#1}
}

\newcommand{\Ric}{
    \text{\normalfont{Ric}}
}
\newcommand{\Rm}{
    \text{\normalfont{Rm}}
}

\newcommand{\vol}{
    \text{\normalfont{vol}}
}

\renewcommand{\Re}{
    \text{\normalfont{Re}\,}
}
\renewcommand{\Im}{
    \text{\normalfont{Im}\,}
}

\newcommand{\lot}{
    \text{\textit{l.o.t.}}
}

\newcommand{\grad}{
    \text{\normalfont{grad}}\,
}
\renewcommand{\div}{
    \text{\normalfont{div}}\,
}
\newcommand{\curl}{
    \text{\normalfont{curl}}\,
}

\newcommand\numberthis{\addtocounter{equation}{1}\tag{\theequation}}

\theoremstyle{plain}
\newtheorem{thm}{Theorem}[section]
\newtheorem{lem}[thm]{Lemma}

\newtheorem{propn}[thm]{Proposition}
\newtheorem{cor}[thm]{Corollary}

\theoremstyle{definition}

\newtheorem{rmk}[thm]{Remark}

\theoremstyle{remark}

\numberwithin{equation}{section}

\title{Flows of conformally coclosed $G_2$-structures with dilaton}

\author{Spiro Karigiannis\thanks{Department of Pure Mathematics, University of Waterloo, \texttt{karigiannis@uwaterloo.ca}} \and
S\'ebastien Picard\thanks{Department of Mathematics, The University of British Columbia, \texttt{spicard@math.ubc.ca}} 
\and Caleb Suan\thanks{Department of Mathematics, The Chinese University of Hong Kong, \texttt{kwsuan@math.cuhk.edu.hk}}
}
\date{}

\begin{document}

\maketitle

\begin{abstract}
We study flows of $G_2$-structures guided by the principle of dimensional reduction: natural geometric flows in $G_2$-geometry reduce to natural flows in complex geometry. Our main examples are the $G_2$-Laplacian coflow, which lifts the K\"ahler--Ricci flow, and a 7-dimensional lift of the anomaly flow on complex threefolds. The $G_2$-lift of the anomaly flow deforms conformally coclosed $G_2$-structures. We compare the $G_2$-anomaly flow to the $G_2$-Laplacian coflow, and investigate short-time existence and fixed points.
\end{abstract}

\tableofcontents

\section{Introduction}
Hamilton’s formulation of the Ricci flow in 1982~\cite{Hamilton1982} marked the beginning of the theory of geometric flows. The Ricci flow equation deforms a metric tensor $g_{ij}$ by the equation
\[
\partial_t g_{ij} = - 2 \Ric_{ij}.
\]
When additional geometric structures are present, there are other geometric flows adapted to preserve the specified particular geometry. The setup for the present work is an oriented 7-dimensional manifold $M^7$ equipped with a positive 3-form $\varphi \in \Omega^3(M^7)$. The theory of $G_2$-geometry associates to such a $\varphi$ a unique metric tensor $g_\varphi$. We  consider natural geometric flows of $G_2$-structures $\varphi_t$ which generate flows of the metric geometry $g_{\varphi_t}$.

There is a general theory of flows of $G_2$-structures~\cite{Kar09,DGK25}. To identify specific flows, we propose the principle of dimensional reduction: natural flows of $G_2$-structures should reduce to natural flows of complex geometry under ans\"{a}tze such as $M^7=S^1 \times {\rm CY}3$ or $M^7= T^3 \times {\rm CY}2$, where $\mathrm{CY}n$ denotes a Calabi--Yau $n$-fold. In this article, we consider two flows in complex geometry, the K\"ahler--Ricci flow and the anomaly flow, and study their $G_2$-lifts. The $G_2$-lift of the K\"ahler--Ricci flow is the coflow~\cite{PS}. The $G_2$-lift of the anomaly flow is a flow of conformally coclosed structures~\cite{AMP24}. This statement is slighly imprecise as both $G_2$-flows can be modified and still descend to their corresponding flows of complex geometry; a more precise statement is given in Remark~\ref{lift-modifiedaf}.

As the coflow is better-known in the literature, let us postpone its further discussion and set our attention for now on the $G_2$-lift of the anomaly flow. The anomaly flow was first proposed in~\cite{PPZ1, PPZ2} as a heat flow to geometrize non-K\"ahler Calabi--Yau threefolds by equations of string theory. A $G_2$-version of the anomaly flow was then derived in~\cite{AMP24} by considering the supersymmetric constraints in 10-dimensional supergravity theories of the form $\mathbb{R}^{3,1} \times M^7$. We restrict here to the case when $\alpha'=0$; otherwise the flow involves more complicated nonlinearities such as $\alpha' \, \tr  R^2$, where $R$ is the Riemann curvature tensor, which we set aside for the present analysis.

The $G_2$-anomaly flow at order $\alpha'=0$ is a flow of conformally coclosed $G_2$-structures with conformal factor given by a dilaton function. Let $(f_0, \varphi_0) \in C^\infty(M^7) \times \Omega^3(M^7)$ be initial data where $\varphi_0$ is a positive 3-form satisfying the conformally coclosed condition
\begin{equation*}
d (e^{-2f_0} \psi_0)= 0.
\end{equation*}
Here we use the standard notation $\psi = \star_{g_\varphi} \varphi$. The dilaton coflow, obtained by setting $\alpha'=0$ in the $G_2$-anomaly flow, is given by
\begin{equation} \label{defn-g2-af}
  \partial_t (e^{-2f} \psi) = d \left( e^{2f} d^* (e^{-2f} \psi) \right) - \frac{2}{3} d \left( ( {\rm tr} \, T) \varphi \right),
\end{equation}
where $T$ is the torsion tensor of $\varphi$, and the dilaton function $f$ is defined by
\[
e^{4f} = e^{4 f_0} \frac{{\rm vol}_\varphi}{{\rm vol}_0}.
\]
The only unknown in the flow is the positive 3-form $\varphi(t)$, as this determines the metric $g(t)$, the 4-form $\psi(t)$, the adjoint $d^*_{g(t)}$, the torsion tensor $T(t)$, and the scalar function $f(t)$. Such a flow satisfies $d (e^{-2f(t)} \psi(t))= 0$ for all times of existence, so that this is a flow of conformally coclosed $G_2$-structures.

The flow of $G_2$-structures~\eqref{defn-g2-af} generates a flow of the corresponding metric tensors $g_{\varphi(t)}$. The result is
   \begin{equation} \label{GRF}
     \partial_t g_{ij} = e^{2f} \Big[ - 2 \Ric_{ij} - 4 \nabla_i \nabla_j f + \frac{1}{2} H_{imn} H_{jmn} \Big],
   \end{equation}
   where the 3-form $H \in \Omega^3(M^7)$ associated to $\varphi$ is
   \begin{equation} \label{eq:H}
H = - e^{2f} d^* (e^{-2f} \psi) + \frac{2}{3} \left( ( {\rm tr} \, T) \varphi \right).
   \end{equation}
   Thus we find a Ricci flow which couples to a scalar function $f$ and a 3-form $H$. The right-hand side of~\eqref{GRF} is well-known in string theory: it is the Einstein equation for the bosonic string. We note that there are also other flows in different special geometries which lead to metric evolution equations of the form~\eqref{GRF}, such as the pluriclosed flow~\cite{StreetsTian2012} on complex manifolds and anomaly flow~\cite{Picard2022} on Calabi--Yau threefolds. There is also the renormalization group flow~\cite{OSW2006} in Riemannian geometry which evolves the metric by~\eqref{GRF} without the conformal factor of $e^{2f}$.

   \begin{rmk}
Fixed points of the dilaton coflow~\eqref{defn-g2-af} solve the generalized Ricci soliton equation. This is a triple $(g,H,f)$, where $g$ is a metric tensor, $H$ is a 3-form and $f$ is a scalar function such that
\begin{equation} \label{gen-soliton}
\begin{aligned}
       -2 \Ric - 4 \operatorname{Hess}(f) + \frac{1}{2} H^2 &= 0, \\
       d H &= 0, \\
       d^* (e^{-2f} H) &=0.
\end{aligned}
\end{equation}
In string theory, these equations are derived by imposing Weyl invariance on a nonlinear sigma model; see for example~\cite{CallanFriedanMartinecPerry1985} or~\cite[Equation (3.7.14a, 3.7.14b)]{Polchinski1998}. In the study of geometric flows, these are the equations of the steady solitons of the generalized Ricci flow~\cite{GFS21}. Special solutions to this system can be constructed by complex geometry and pluriclosed metrics via the pluriclosed flow~\cite{StreetsTian2013}, and classification of solutions on compact manifolds of dimension four was carried out in~\cite{Streets2019, StreetsYuri2021, ASU2023} when the solutions come from a generalized K\"ahler structure. We remark that~\eqref{defn-g2-af} is an odd-dimensional flow arriving at the same system of equations; but in contrast to the generalized Ricci flow, the condition $dH=0$ appears in the limit rather than along the flow lines.
  \end{rmk}
 
The theme, then, is to find special solutions to the equation of Riemannian geometry~\eqref{gen-soliton}. We see that when $H=0$ and $f$ is constant then~\eqref{gen-soliton} is the Ricci-flat equation, while when $H=0$ this is the gradient steady soliton equation. The flow~\eqref{defn-g2-af} is a way to produce special solutions to~\eqref{gen-soliton} in dimension seven, just as the Laplacian flow is a way to produce special solutions to the Ricci-flat equation in dimension seven, and as the pluriclosed flow produces special solutions to~\eqref{gen-soliton} in even dimensions. 

Though derived from natural reasoning in string theory, there is a problem: it remains open to prove that the flow~\eqref{defn-g2-af} is well-defined. This situation is reminiscent of the $G_2$-coflow, whose short-time existence in general is also still open. Inspired by Grigorian's~\cite{Gri13} modification of the coflow, we propose here to modify the $G_2$-anomaly flow with a similar additional term, and we are able to prove short-time existence of the modified $G_2$-anomaly flow. 
\begin{thm} \label{thm:STE-intro}
  Let $M^7$ be a compact 7-manifold with initial positive 4-form $\psi_0$ satisfying the conformally coclosed condition $d ( e^{-2f_0} \psi_0)=0$ for some function $f_0$. If $C<1$, then the flow
  \begin{equation} \label{defn-g2-af2}
\partial_t (e^{-2f} \psi) = d \left( e^{2f} d^* (e^{-2f} \psi) \right) + (C-2) d \left( ({\rm tr} \, T) \varphi \right),
  \end{equation}
 where $e^{4f} = e^{4f_0} \frac{\vol_\varphi}{\vol_0}$, admits a unique short-time solution. Note that for all time for which the solution exists, the cohomology class of $e^{-2 f} \psi$ is preserved.
\end{thm}

Note that the dilaton coflow~\eqref{defn-g2-af} corresponds to $C=\frac{4}{3}$ in~\eqref{defn-g2-af2}, which is not covered by our short-time existence result. (Note that we sometimes refer to the dilaton coflow as the $G_2$-anomaly flow in the paper, but we are always taking $\alpha' = 0$ throughout.)

\begin{rmk}
 This result may be compared with the corresponding theory for the $G_2$-coflow. Let us recall the setup. The Laplacian coflow, introduced in~\cite{KMP12} and modified in~\cite{Gri13}, is defined by the following equations:
    \begin{equation} \label{defn-coflow}
\partial_t \psi = d d^* \psi + (C-2) d \left( ({\rm tr} \, T) \varphi \right), \quad d \psi(0)=0.
\end{equation}
It is proved in~\cite{Gri13} that if $C<1$, then the flow admits a unique short-time solution. 
\end{rmk}

The modified flows~\eqref{defn-g2-af2}, for any constant $C$, are all valid $G_2$-lifts of the anomaly flow on Calabi--Yau threefolds with $\alpha'=0$. This is proved in Remark~\ref{lift-modifiedaf}, and the result is that the anomaly flow on ${\rm CY}3$ generates a solution to~\eqref{defn-g2-af2} on $S^1 \times {\rm CY}3$ for any constant $C$. This observation justifies calling~\eqref{defn-g2-af2} the modified $G_2$-anomaly flow.

Next, we analyze the fixed points of the modified flow~\eqref{defn-g2-af2}. We find the following result.

\begin{thm} \label{thm:fixed-points-intro}
    The fixed points $\varphi$ of the modified $G_2$-anomaly flow
    \begin{align}
        \del_t (e^{-2f} \psi) =  d \big( e^{2f} d^* (e^{-2f} \psi) \big) + (C-2) d \big( (\tr T) \varphi \big), \tag{\ref{eqn-G2-Anomaly-flow}}
    \end{align}
    satisfy the following:
    \begin{itemize} \setlength\itemsep{-1mm}
        \item if $C \neq 1$, then
        \begin{align*}
            \tau_0 = K \exp \Big( - \frac{3 (C - \frac{4}{3})}{2 (C-1)} f \Big)
        \end{align*}
        for some constant $K$, where $\tau_0 = \frac{4}{7} ({\rm tr} \, T)$;
        
        \item if $C = \frac{4}{3}$, then $\tau_0$ is constant and $(g,H,f)$ is a generalized Ricci soliton~\eqref{gen-soliton};

        \item if $C = \frac{10}{7}$, then $f$ is constant and $\varphi$ is nearly parallel or torsion-free;

        \item if $C > \frac{10}{7}$, then $f$ is constant and $\varphi$ is torsion-free.
    \end{itemize}
  \end{thm}

Classification of the fixed points when $C<1$ remains open. When $C> \frac{10}{7}$, the fixed points are torsion-free, with the caveat that this is outside the current short-time existence range. The constant $C=\frac{4}{3}$ distinguished by string theory appears as a special case where the fixed points are not necessarily torsion free, but $\tau_0$ must be constant and the generalized Ricci soliton~\eqref{gen-soliton} equation holds.

\begin{rmk}
    After our paper was first posted to the arXiv in November 2025, the authors were informed of similar results obtained independently by Garcia-Fernandez--Moreno--Payne--Streets in \cite{GFMPS25}. Their results also consider a similar family of flows that evolve conformally coclosed $G_2$-structures with dilaton, however they decouple the dilaton from the $G_2$-structures and augment the lower-order terms in the evolution of the dilaton which force fixed points to be torsion-free for particular choices of certain parameters.
\end{rmk}

The paper is organized as follows. Our calculations and derivations take place in the setting of the modified $G_2$-anomaly flow, and since this is a flow of conformally coclosed $G_2$-structures, along the way we occasionally set the conformal factor to be constant and extract statements for the modified $G_2$-coflow. Section~\ref{sect-lifting} surveys various results on $G_2$-lifts of well-known geometric flows, and establishes the motivation for our analysis on coflows and conformally coclosed flows. Section~\ref{sect-curvature} derives various identities and an expression for the Ricci curvature of a conformally coclosed $G_2$-structure. Section~\ref{sect-modified-G2-Anomaly-flows} examines the fixed points and the metric evolution for the $G_2$-anomaly flow, and Section~\ref{sect-ste} is concerned with short-time existence and uniqueness. Finally, in Section~\ref{sect-summary}, we summarize and compare our analysis of fixed points and short-time existence with what is known for the modified $G_2$-coflow~\cite{FF25}.

\subsubsection*{Notation and conventions}

Any local computations are done with respect to a local orthonormal frame. As such, all indices are subscripts and the Einstein summation convention sums over repeated subscripts.

We summarize here our conventions for $G_2$-geometry, which follow those of~\cite{Kar09, DGK25}. Let $M^7$ be an oriented 7-manifold. We say $\varphi \in \Omega^3(M^7)$ is a $G_2$-structure if it determines a metric tensor $g_\varphi$ via the formula
\begin{equation*}
( X \hook \varphi) \wedge (Y \hook \varphi) \wedge \varphi = -6 g_\varphi(X,Y) \vol_{g_\varphi}.
\end{equation*}
The metric $g_\varphi$ determines a Hodge star $\star_\varphi$, and we use the notation
\begin{equation*}
\psi = \star_\varphi \varphi, \quad \psi \in \Omega^4(M^7).
\end{equation*}
We say a $G_2$-structure is closed if $d \varphi = 0$ and coclosed if $d \psi = 0$. 

Differential forms on a manifold equipped with a $G_2$-structure admit a type decomposition:
\begin{align*}
  \Omega^2 & = \Omega^2_{27} \oplus \Omega^2_{14}, \\
  \Omega^3 &= \Omega_1^3 \oplus \Omega^3_7 \oplus \Omega^3_{27},
\end{align*}
and decompositions for $\Omega^4$ and $\Omega^5$ are obtained by the Hodge star. The decompositions are given by
\begin{align*}
  \Omega^2_7 &= \{ V \hook \varphi : V \in \Gamma(TM) \}, \\
  \Omega^2_{14} &= \{ \beta \in \Omega^2 : \beta \wedge \psi = 0 \}, \\
  \Omega^3_1 &= \{ f \varphi : f \in C^\infty(M) \}, \\
  \Omega^3_7 &= \{ V \hook \psi : V \in \Gamma(TM) \}, \\
  \Omega^3_{27} &= \{ \gamma \in \Omega^3 \: \gamma \wedge \varphi = \gamma \wedge \psi = 0 \}.
\end{align*}

The above decomposition of $2$-forms in turn allows us to decompose the entire space $\cal{T}^2$ of $2$-tensors as
\begin{align*}
    \cal{T}^2 = C^\infty g \oplus \cal{S}^2_0 \oplus \Omega^2_7 \oplus \Omega^2_{14},
\end{align*}
where $\cal{S}^2_0$ is the space of traceless symmetric $2$-tensors. We use the subscripts $1$, $27$, $7$, and $14$ to refer to the projection of a $2$-tensor onto each of these four respective spaces.

Contracting a $2$-tensor $A$ with $\varphi$ using the metric tensor by
\begin{align*}
    (A \diamond \varphi)_{ijk} = A_{im} \varphi_{mjk} - A_{jm} \varphi_{mik} + A_{km} \varphi_{mij}
\end{align*}
defines a linear map $\diamond \, \varphi \: \cal{T}^2 \rarr \Omega^3$. In particular, the kernel of this map is the space $\Omega^2_{14}$ and the $\diamond$ map is an isomorphism between
\begin{align*}
    C^\infty g \cong \Omega^3_1, \quad \Omega^2_7 \cong \Omega^3_7, \quad \cal{S}^2_0 \cong \Omega^3_{27}.
\end{align*}
One can similarly define a contraction map $\diamond \, \psi$ which acts on a $2$-tensor $A$ by
\begin{align*}
    (A \diamond \psi)_{ijkl} = A_{im} \psi_{mjkl} - A_{jm} \psi_{mikl} + A_{km} \psi_{mijl} - A_{lm} \psi_{mijk}.
\end{align*}
As is the case for $3$-forms, $\Omega^2_{14}$ is the kernel of this map, and we have isomorphisms
\begin{align*}
    C^\infty g \cong \Omega^4_1, \quad \Omega^2_7 \cong \Omega^4_7, \quad \cal{S}^2_0 \cong \Omega^4_{27}.
\end{align*}
We also record the useful identity
\begin{equation} \label{eq:27-hook}
(X \hook \varphi) \diamond \varphi = - 3 X \hook \psi.
\end{equation}

The \emph{torsion} of a $G_2$-structure is a 2-tensor $T$ defined by
\begin{equation*}
\nabla_m \varphi_{ijk} = T_{mp} \psi_{pijk}.
\end{equation*}
The torsion of a $G_2$-structure is also described by the torsion forms: a function $\tau_0$, a 1-form $\tau_1$, a $\Omega^2_{14}$-form $\tau_2$ and a $\Omega^3_{27}$-form
$$ \tau_3 = \frac{1}{2} (\tau'_3)_{im} \varphi_{mjk} \, e_i \wedge e_j \wedge e_k = \tau_3' \diamond \varphi $$
for some $\tau_3' \in \cal{S}^2_0$ such that
\begin{equation} \label{eq:torsion-forms}
\begin{aligned}
d \varphi & = \tau_0 \psi + 3 \tau_1 \wedge \varphi + \star \tau_3, \\
d \psi & = 4 \tau_1 \wedge \psi + \star \tau_2,
\end{aligned}
\end{equation}
The torsion tensor $T$ and the torsion forms $\tau_i$ are related (see~\cite{Kar09}) by
\begin{equation} \label{eq:torsion-relations}
T_{ij} = \frac{\tau_0}{4} g_{ij} - (\tau'_3)_{ij} + (\tau_1)_p \varphi_{pij} - \frac{1}{2} (\tau_2)_{ij}.
\end{equation}

\subsubsection*{Acknowledgments}
We thank the CRM in Montr\'eal for supporting our residency during the Thematic Program in Geometric Analysis in 2024. We also thank Weiyong He, Jock McOrist, and Eirik Svanes for useful discussions. The research of SK and SP is partially supported by NSERC Discovery Grants.

\section{Lifting flows of complex geometry} \label{sect-lifting}

In this section we explore how flows of complex geometry generate flows of $G_2$-structures. Table~\ref{table:lifts} summarizes the results.

\begin{table}[H]
\centering
\begin{tabular}{|l|l|}
\hline
\textbf{Complex Geometric Flow} & \textbf{\(G_2\)-Flow} \\
\hline
Kähler--Ricci flow               & \( G_2 \)-Laplacian coflow \\
\( {\rm MA}^{1/3} \) flow             & \( G_2 \)-Laplacian flow \\
Anomaly flow ($\alpha'=0$)                   & Conformally coclosed flow \\
\hline
\end{tabular}
\caption{Examples of flows in complex geometry and their \( G_2 \)-lifts. The complex counterpart to the \( G_2 \)-Laplacian flow and coflow is derived in~\cite{PS}, and we derive the complex counterpart to the conformally coclosed flow in Proposition~\ref{lift-af}.}  \label{table:lifts}
\end{table}

\subsection{K\"ahler--Ricci flow}
Our point of departure is the most well-known flow in complex geometry, the K\"ahler--Ricci flow, which is known to lift to the Laplacian coflow in $G_2$-geometry.

 \begin{propn} \label{lift-krf}~\cite{PS}
   Let $M^7= S^1 \times X^6$, where $X^6$ is a compact complex threefold admitting a holomorphic volume form $\Omega$ and a K\"ahler metric $\omega_0$ such that $d \omega_0=0$. Start the K\"ahler--Ricci flow
   \[
{\del_t \tilde{\omega}_t} = - 2 {\rm Ric}(\tilde{\omega}_t)
   \]
   with initial data $\tilde{\omega}_{t=0} = \omega_0$. Then if we let
    \begin{equation} \label{eq:coflow}
\psi_t = - dr \wedge {\rm Im} ( \Omega_t ) - \frac{1}{2} \omega_t^2, \quad (\omega_t, \Omega_t) = (\Theta_t^* \tilde{\omega}_t, \Theta_t^* \Omega),
    \end{equation}
 where $r$ denotes the angle coordinate on the $S^1$, the evolving family of $G_2$-structures $(M^7,\psi_t)$ solves the Laplacian coflow
    \begin{equation} \label{eq:coflow2}
{\del_t \psi_t} =d d^* \psi_t.
\end{equation}
Here $\Theta_t$ is a family of diffeomorphisms determined by $\Theta_0 = id$ and ${\del_t} \Theta_t = \nabla_{\tilde{g}_t} \log |\Omega|_{\tilde{\omega}_t}$.
  \end{propn}

  \begin{proof}
    This is a result from~\cite{PS}, and the full calculation can be found there (albeit with different coefficients). We give the outline. First, compute the derivative of $\omega_t = \Theta_t^* \tilde{\omega}_t$. We have
    \begin{align*}
      {\del_t \omega_t} &= {\del_t} \Theta_t^* \tilde{\omega}_t \\
                            &= \Theta_t^* \mathcal{L}_{\nabla_{\tilde{g}_t} \log |\Omega|_{\tilde{\omega}_t}} \tilde{\omega}_t + \Theta_t^* (- 2 {\rm Ric}(\tilde{\omega}_t)) \\
      &= (1-2) \mathcal{L}_{\nabla \log |\Omega_t|_{\omega_t}} \omega_t
    \end{align*}
    by the K\"ahler identity ${\rm Ric}(\omega) = \mathcal{L}_{\nabla \log |\Omega|_\omega} \omega$. Thus the ansatz is setup such that
    \begin{equation*}
( {\del_t \omega_t}, {\del_t \Omega_t} ) = (- \mathcal{L}_{\nabla \log |\Omega_t|_{\omega_t}} \omega_t, \mathcal{L}_{\nabla \log |\Omega_t|_{\omega_t}} \Omega_t),
\end{equation*}
 and hence, using that the induced metric is $g_t = \frac{1}{8} |\Omega|^2 dr^2 + (\tilde g_X)_t$ from~\cite[equation (2.20)]{PS}, we get
    \[
\del_t \left( - dr \wedge {\rm Im} (\Omega_t) - {\frac{1}{2}} \omega_t^2 \right) = \mathcal{L}_{\nabla \log |\Omega_t|_{\omega_t}} \left( - dr \wedge {\rm Im} (\Omega_t) + {\frac{1}{2}} \omega_t^2 \right).
    \]
    On the other hand, the Laplacian is computed in~\cite{PS} and the result is
    \[
d^* d \left( - dr \wedge {\rm Im} ( \Omega ) - \frac{1}{2} \omega^2 \right) =  \mathcal{L}_{\nabla \log |\Omega|_\omega} \left( - dr \wedge {\rm Im} (\Omega) + {\frac{1}{2}} \omega^2 \right).
    \]
    We see that ansatz~\eqref{eq:coflow} generates a solution to the $G_2$-coflow~\eqref{eq:coflow2}.
  \end{proof}
 
  \begin{rmk}
  Proposition~\ref{lift-krf} also holds for the modified coflow
    \[
\partial_t \psi = d d^* \psi + (C-2) d \Big( ({\rm tr} \, T) \varphi \Big), \quad d \psi(0) = 0
    \]
    where $C$ is arbitrary. This is because ${\rm tr} \, T = 0$ for ansatz~\eqref{eq:coflow}. See also~\cite{EarpSaavedraSuan2025} for further extensions.
  \end{rmk}
  
  \begin{rmk}
    The K\"ahler--Ricci flow on $X^6$ does not generate solutions to the $G_2$-Laplacian flow on $S^1 \times X^6$; instead, the correct flow in that setup is the ${\rm MA}^{1/3}$-flow. We refer to~\cite{PS} for further details on deriving Monge--Amp\`ere flows from the $G_2$-Laplacian flow. 
  \end{rmk}

\subsection{Anomaly flow}

The next flow of complex geometry that we will consider is anomaly flow with $\alpha'=0$ on a Calabi--Yau threefold; we refer to~\cite{PPZ1, PPZ2} for background on the anomaly flow. In this section we lift this Hermitian flow to a flow of $G_2$-geometry. The lifted equation is a flow of conformally coclosed $G_2$-structures which resembles the $G_2$-Laplacian coflow. 

First, recall the anomaly flow with $\alpha'=0$ on a Calabi--Yau threefold. This is a flow of conformally coclosed structures in 6-dimensions. Let $X^6$ be a complex threefold admitting a holomorphic volume form $\Omega \in \Omega^3(X^6,\mathbb{C})$. The flow equations for a hermitian metric $\omega$ are:
\[
\partial_t ( e^{-2f} \star \omega) = 2 i \partial \bar{\partial} \omega
\]
where the dilaton function is defined by
\[
  e^{-2f}= \frac{i \Omega \wedge \bar{\Omega}}{\frac{1}{3!} \omega^3}.
\]
The holomorphic volume form $\Omega$ is held fixed while the hermitian form $\omega$ evolves. If the initial condition $\omega_0$ is conformally coclosed, then
\[
d ( e^{-2f} \star \omega) = 0
\]
along the flow. As proved in~\cite{PPZ1, PPZ3}, this flow always admits a short-time solution. Using the common notation $e^{-2f} = |\Omega|_\omega$ for the norm of a holomorphic $(3,0)$-form, the flow can also be rewritten as
\[
\partial_t ( |\Omega|_{\omega_t} \, \omega_t^2) = 4 i \partial \bar{\partial} \omega_t, \quad d(|\Omega|_{\omega_0} \, \omega_0^2) = 0.
\]
This conformally coclosed flow in dimension six can be lifted to a conformally coclosed flow in dimension seven.

\begin{propn} \label{lift-af}
   Let $M^7= S^1 \times X^6$, where $X^6$ is a compact complex threefold admitting a holomorphic volume form $\Omega$ and a hermitian metric $\omega_0$ such that $d (|\Omega|_{\omega_0} \, \omega_0^2) = 0$. Start the $\alpha'=0$ anomaly flow, which flows the hermitian metrics $\omega_t$ by
   \[
\partial_t ( |\Omega|_{\omega_t} \, \omega_t^2) = 4 i \partial \bar{\partial} \omega_t
   \]
 with initial data $\omega_0$. Then if we let
   \begin{equation} \label{af:ansatz}
\psi = - dr \wedge \frac{1}{|\Omega|_{\omega_t}} {\rm Im} \, \Omega  - \frac{1}{2} \omega_t^2,
    \end{equation}
    where $r$ denotes the angle coordinate on the $S^1$, the evolving family of $G_2$-structures $(M^7,\psi)$ solves the conformally coclosed flow
    \begin{equation} \label{eq:confcoflow}
\partial_t ( e^{-2f} \psi) = d \left( e^{2f} d^* ( e^{-2f} \psi) \right), \quad e^{4 f} = \frac{{\rm vol}_\psi}{{\rm vol}_R},
\end{equation}
where the fixed reference volume ${\rm vol}_R \in \Omega^7(M^7)$ is chosen to be ${\rm vol}_R = dr \wedge (i \Omega \wedge \bar{\Omega})$.
\end{propn}

  \begin{proof}

We start by deriving identities for the geometry generated by a $G_2$-structure $\varphi$ on $M = S^1 \times X^6$ of the form
\begin{equation*}
    \varphi = 2 \sqrt{2} \Re \Big( \frac{1}{ |\Omega |_\omega} \Omega \Big) - \frac{1}{2 \sqrt{2}} dr \wedge \omega,
\end{equation*}
where $r$ is the angle coordinate on the $S^1$-component. (The coefficients on the terms are to achieve a nicer expression for the dual $4$-form $\psi$.) Our conventions on $(X^6,g,\omega,\Omega)$ are:
\begin{equation*}
\vol_6 = \frac{\omega^3}{3!} = \frac{i}{|\Omega|^2_\omega} \Omega \wedge \bar{\Omega}.
\end{equation*}
The $G_2$-structure $\varphi$ above creates a metric tensor $g_\varphi$, a volume form $\vol_\varphi$, a Hodge star $\star_\varphi$ and a dual $4$-form $\psi$. Indeed, the metric tensor and volume form are given by
\begin{align*}
  g_\varphi &= \frac{1}{8} dr^2 + g_6, \\
    \vol_\varphi &= \frac{1}{2 \sqrt{2}} dr \wedge \vol_6,
\end{align*}
and we can compute that for $\beta \in \Omega^k(X^6)$, we have
\begin{align*}
    \star_\varphi \beta &= (-1)^k \frac{1}{2 \sqrt{2}} dr \wedge \star_6 \beta, \\
    \star_\varphi (dr \wedge \beta) &= 2 \sqrt{2} \star_6 \beta.
\end{align*}
Thus $\psi = \star_\varphi \varphi$ gives
\begin{equation*}
    \psi = - dr \wedge \Im \Big( \frac{1}{ |\Omega |_\omega} \Omega \Big) - \frac{1}{2} \omega^2.
  \end{equation*}
By~\eqref{eq:confcoflow}, the conformal factor $f$ for $\varphi$ is defined to be
\begin{equation*}
  f = \frac{1}{4} \log \frac{\vol_\varphi}{\vol_R} = - \frac{1}{2} \log |\Omega|_{\omega},
\end{equation*}
which gives $|\Omega|_{\omega} = e^{-2 f}$, so the conformally balanced equation $d (|\Omega|_{\omega} \omega^2)= 0$ on $X^6$ becomes $d (e^{-2f} \psi) = 0$ on $M^7$.

With this setup in place, we can now substitute ansatz~\eqref{af:ansatz} into the flow equation~\eqref{eq:confcoflow}. The left-hand side of~\eqref{eq:confcoflow} is
\begin{equation}
    \del_t (e^{-2f} \psi) = \del_t \Big( -dr \wedge \Im \Omega - \frac{1}{2} |\Omega |_\omega \omega^2 \Big) = - \frac{1}{2} \del_t ( |\Omega |_\omega \omega^2).
\end{equation}
The right-hand side of~\eqref{eq:confcoflow} is
\begin{equation}
\begin{aligned}
    d (e^{2f} d^*_\varphi (e^{-2f} \psi)) &= d \Big( e^{2f} \star_\varphi d \star_\varphi \Big( -dr \wedge \Im \Omega - \frac{1}{2} |\Omega|_\omega \omega^2 \Big) \Big) \\
    &= d \Big( e^{2f} \star_\varphi d \Big( 2 \sqrt{2} \Re \Omega - \frac{1}{2 \sqrt{2}} |\Omega|_\omega dr \wedge \omega \Big) \Big) \\
    &= d \Big( e^{2f} \star_\varphi \frac{1}{2 \sqrt{2}} \Big( dr \wedge d |\Omega|_\omega \wedge \omega + |\Omega|_\omega dr \wedge d \omega \Big) \Big) \\
    &= d \Big( e^{2f} \Big( \star \Big[ d |\Omega|_\omega \wedge \omega \Big] + |\Omega|_\omega \star d\omega \Big) \Big),
\end{aligned}
\end{equation}
and so
\begin{equation} \label{rhs-ansatz}
 d (e^{2f} d^*_\varphi (e^{-2f} \psi)) =  d \star \Big[ d (\log |\Omega|_\omega) \wedge \omega \Big] + d \star d\omega.
  \end{equation}
Let us write 
\begin{equation}
    d\omega = \mu \wedge \omega + (d\omega)_0
\end{equation}
where $\mu$ is the Lee form and $(d\omega)_0$ is the primitive part of $d\omega$. We claim the following two identities: first, we have
\begin{equation} \label{cplx-geo-id1}
 d \star d \omega =  -2 d \Big[ (J \mu) \wedge \omega \Big] - 2 i \del \delbar \omega,
\end{equation}
and, since $d (|\Omega|_\omega \, \omega^2) = 0$, we also have
\begin{equation} \label{cplx-geo-id2}
   d (\log |\Omega|_\omega) = -2 \mu.
\end{equation}
Let us assume~\eqref{cplx-geo-id1} and~\eqref{cplx-geo-id2} for now and substitute these into~\eqref{rhs-ansatz}.  The result is
\begin{align*}
    d (e^{2f} d^*_\varphi (e^{-2f} \psi)) &= -2 d \star (\mu \wedge \omega) -2 d \Big[ (J \mu) \wedge \omega \Big] - 2 i \del \delbar \omega \\
    &= 2 d \Big[ (J \mu) \wedge \omega \Big] -2 d \Big[ (J \mu) \wedge \omega \Big] - 2 i \del \delbar \omega \\
    &= -2 i \del \delbar \omega.
\end{align*}
Therefore the flow
\begin{equation*}
    \del_t (e^{-2f} \psi) = d (e^{2f} d^*_\varphi (e^{-2f} \psi))
\end{equation*}
does indeed reduce to
\begin{equation*}
    \frac{1}{2} \del_t (|\Omega|_\omega \omega^2) = 2 i \del \delbar \omega
\end{equation*}
on the base $X^6$.

We now give the detailed justification of~\eqref{cplx-geo-id1} and~\eqref{cplx-geo-id2}. For~\eqref{cplx-geo-id1}, we use the fact that in complex dimension $3$, the Hodge star acts by
\begin{equation*}
    \star \gamma = J \gamma
\end{equation*}
for a primitive $3$-form $\gamma \in \cal{P}^3$ and
\begin{equation*}
    \star (\alpha \wedge \omega) = - J (\alpha \wedge \omega) = - (J \alpha) \wedge \omega
\end{equation*}
for a $1$-form $\alpha \in \Omega^1$. (See, for example,~\cite[Proposition 1.2.31]{Huybrechts}.) Thus we have
\begin{align*}
    \star d \omega &= \star (\mu \wedge \omega) + \star (d\omega)_0 \\
    &= - (J \mu) \wedge \omega + J (d\omega - \mu \wedge \omega) \\
    &= - 2 (J \mu) \wedge \omega + i (\del \omega - \delbar \omega).
\end{align*}
which proves~\eqref{cplx-geo-id1}. For~\eqref{cplx-geo-id2}, we use that the anomaly flow preserves the conformally balanced condition, so that
\begin{equation*}
 d (|\Omega|_{\omega_t} \omega_t^2) = 0
\end{equation*}
for all time. Expanding out this equation gives
\begin{align*}
    0 &= d |\Omega|_\omega \wedge \omega^2 + 2 |\Omega|_\omega \omega \wedge d\omega \\
    &= d |\Omega|_\omega \wedge \omega^2 + 2 |\Omega|_\omega \omega \wedge \mu \wedge \omega.
\end{align*}
Dividing both sides by $|\Omega|_\omega$, we get
\begin{equation*}
    0 = d (\log |\Omega|_\omega) \wedge \omega^2 + 2 \mu \wedge \omega^2.
\end{equation*}
Since wedging with $\omega^2$ is an isomorphism on 1-forms, we obtain~\eqref{cplx-geo-id2}.
\end{proof}

\begin{rmk} \label{lift-modifiedaf}
  Proposition~\ref{lift-af} also holds for the modified $G_2$-anomaly flows. Namely, the $\alpha'=0$ anomaly flow on $(X^6,\omega_t)$ generates via ansatz~\eqref{af:ansatz} a solution to the equations
  \begin{equation} \label{modifiedconfcoclosed}
\partial_t ( e^{-2f} \psi) = d \left( e^{2f} d^* ( e^{-2f} \psi) \right) + (C-2) d \Big( ({\rm tr} \, T) \varphi \Big), \quad d (e^{-2f} \psi) = 0
    \end{equation}
    where $C$ is arbitrary. This is because ${\rm tr} \, T = 0$ for ansatz~\eqref{af:ansatz}. Indeed, we note that by type decomposition,
\begin{align}
    \varphi \wedge d\varphi &= \Big[ \underbrace{\Re \Big( \frac{1}{|\Omega |_\omega} \Omega \Big)}_{(3,0) \oplus (0,3)} - \underbrace{dr \wedge \omega}_{dr \wedge (1,1)} \Big] \wedge \Big[ - \underbrace{ d (\log |\Omega|_\omega) \wedge \Re \Big( \frac{1}{|\Omega|_\omega} \Omega \Big)}_{(3,1) \oplus (1,3)} + \underbrace{dr \wedge d\omega}_{dr \wedge [(2,1) \oplus (1,2)]} \Big] = 0
\end{align}
and so ${\rm tr} \, T = 0$ for this $G_2$-structure. Therefore all modified $G_2$-anomaly flows coincide with the $SU(3)$-anomaly flow on $X^6$. Additional arguments from string theory~\cite{AMP24} distinguish the particular constant $C = \frac{4}{3}$ among these flows.
  \end{rmk}

  In summary, we have demonstrated that the conformally coclosed flow~\eqref{modifiedconfcoclosed} can be understood as a $G_2$-anomaly flow with $\alpha'=0$. Though the anomaly flow on a Calabi--Yau threefold~\cite{PPZ1, PPZ2} involves additional $\alpha'$-terms, we do not consider these terms in the $G_2$-anomaly flow in the present work (see~\cite{AMP24} for progress in this direction).

\begin{rmk}
Another natural flow of non-K\"ahler structures is the pluriclosed flow~\cite{StreetsTian}. We leave open the question of deriving the $G_2$-lifts of the pluriclosed flow. 
\end{rmk}

\subsection{Ascension of the complex Monge--Amp\`ere equation} \label{sec:ascension}

Having derived the $G_2$-flow
\[
\partial_t ( e^{-2f} \psi) = d \left( e^{2f} d^* ( e^{-2f} \psi) \right)
\]
from a 6-dimensional coclosed flow, we now construct special solutions by lifting complex Monge--Amp\`ere flows. Let $(X,g)$ be a compact K\"ahler manifold and $a(x) \in C^\infty(X)$ a given function. Consider the Monge--Amp\`ere flow
    \begin{equation} \label{CMA-flow}
{\del_t u} = e^{-a(x)} \frac{\det (g_{\alpha \bar{\beta}} + u_{\alpha \bar{\beta}})}{\det g_{\alpha \bar{\beta}} }, \quad g_{\alpha \bar{\beta}} + u_{\alpha \bar{\beta}} > 0.
      \end{equation}
This Monge--Amp\`ere flow is not the K\"ahler--Ricci flow~\cite{Cao1985} because there is no logarithm. It was shown in~\cite{PPZ3, PicardZhang2020} that the flow $u(t)$ exists for all $t>0$, and the deviation from the average $\tilde{u} = u - \frac{1}{V} \int_X u \, d {\rm vol}_g$ converges smoothly as $t \rightarrow \infty$. Here $V$ is the volume of $(X,g)$. The proof involves a modification of Yau's~\cite{Yau78} a priori estimates for the complex Monge--Amp\`ere equation. One particular difficulty is that the Monge--Amp\`ere determinant is not concave, and so the general theory of~\cite{PhongTo2021} does not apply.

      We now restrict to the setting of a K\"ahler manifold $(X,g)$ of dimension 3 with a holomorphic volume form $\Omega$, and set $a(x) = \log |\Omega|^2_g - \log 2$. The complex Monge--Amp\`ere flow~\eqref{CMA-flow}  becomes
      \begin{equation} \label{CMA-flow2}
{\del_t u} = 2|\Omega|^{-2}_{\chi}, \quad \chi_{\alpha \bar{\beta}} = g_{\alpha \bar{\beta}} + u_{\alpha \bar{\beta}}>0
        \end{equation}
        where $\chi \in \Omega^{1,1}(X,\mathbb{R})$. It was noticed in~\cite{PPZ3} that a solution to the Monge-Amp\`ere flow~\eqref{CMA-flow2} creates a solution to the equation
      \begin{equation} \label{6dconformal}
\del_t ( |\Omega|_\omega \omega^2) = 4 i \partial \bar{\partial} \omega.
      \end{equation}
      We give here an outline of the calculation. First, from a solution to~\eqref{CMA-flow}, we build a conformally coclosed hermitian metric $\omega(t)$ by the ansatz
      \[
|\Omega|_{\omega(t)} \omega(t)^2 = \chi(t)^2, \quad \chi_{\alpha \bar{\beta}} = g_{\alpha \bar{\beta}} + u_{\alpha \bar{\beta}}.
\]
Solving for $\omega$ in this ansatz gives $\omega = |\Omega|_\chi^{-2} \chi$. (See~\cite{PPZ3} for details.) Next, we substitute the ansatz for $\omega(t)$ into the flow~\eqref{6dconformal}. The result is
      \[
2 i \partial \bar{\partial} \dot{u} \wedge \chi = 4 i \partial \bar{\partial} (|\Omega|^{-2}_\chi) \wedge \chi.
      \]
      We see that since $u$ is a solution to the Monge--Amp\`ere flow~\eqref{CMA-flow2}, then $\omega(t)$ is a solution to the conformal coclosed flow~\eqref{6dconformal}.

We now make the jump from six to seven dimensions.

\begin{propn} Let $(X^6,g)$ be a compact K\"ahler manifold of dimension 3 with holomorphic volume form $\Omega$. Let $u$ be a solution to
 \[
{\del_t u} = e^{-a(x)} \frac{\det (g_{\alpha \bar{\beta}} + u_{\alpha \bar{\beta}})}{\det g_{\alpha \bar{\beta}} }, \quad \chi_{\alpha \bar{\beta}} = g_{\alpha \bar{\beta}} + u_{\alpha \bar{\beta}} > 0,
   \]
with $a(x) = \log |\Omega|^2_g-\log 2$. Then if we let $M^7 = S^1 \times X^6$ and 
   \begin{equation} 
\psi = - dr \wedge \frac{1}{|\Omega|_{\omega}} {\rm Im} \, \Omega  - \frac{1}{2} \omega^2, \quad \omega = |\Omega|_\chi^{-2} \chi,
    \end{equation}
    where $r$ denotes the angle coordinate on the $S^1$, the resulting $G_2$-structures on $M^7$ solve the flow equations
    \[
      \partial_t ( e^{-2f} \psi) = d \left( e^{2f} d^* ( e^{-2f} \psi) \right) + (C-2) d \Big( ({\rm tr} \, T) \varphi \Big)
      \]
    for any constant $C$. Here $e^{4 f} = \frac{{\rm vol}_\psi}{{\rm vol}_R}$ with ${\rm vol}_R = dr \wedge (i \Omega \wedge \bar{\Omega})$. 
  \end{propn}

The proof is a consequence of Proposition~\ref{lift-af} applied to~\eqref{6dconformal}, together with Remark~\ref{lift-modifiedaf}. We conclude that the conformally coclosed $G_2$-flows considered in this paper are generalizations of complex Monge--Amp\`ere flows in K\"ahler geometry.

\subsection{Remarks on the hypersymplectic flow} \label{sec:HS} 

We have thus far considered $G_2$-lifts from a 6-manifold $X^6$ to $M^7=S^1 \times X^6$. We briefly comment on lifting geometric flows on a 4-manifold $X^4$ to 7-dimensions via $M^7 = T^3 \times X^4$. Examples of this include:

  \begin{itemize} \setlength\itemsep{-1mm}
  \item The K\"ahler--Ricci flow on a $K3$ surface once again generates the coflow on $M^7 = T^3 \times K3$ via a suitable ansatz; we omit the details and refer to~\cite{PS}.
  \item There is a natural flow of hypersymplectic structures on a 4-manifold which lifts to the $G_2$-Laplacian flow~\cite{FineYao18, FineYao19}.
  \end{itemize}

Let us say a few words on the flow of hypersymplectic structures which lifts to the $G_2$-Laplacian flow. Following Donaldson~\cite{Don06} and Fine--Yao~\cite{FineYao18}, a hypersymplectic structure on an oriented 4-manifold $X^4$ is a triple of 2-forms $\underline{\omega} = (\omega_1,\omega_2,\omega_3)$ such that each $\omega_i \in \Omega^2(X^4)$ is symplectic, and furthermore $\alpha \omega_1 + \beta \omega_2 + \gamma \omega_3$ is symplectic for all $(\alpha, \beta, \gamma) \neq (0,0,0)$. From a hypersymplectic structure $\underline{\omega}$, we obtain a metric $g_{\underline{\omega}}$ tensor on $X^4$ via
  \begin{equation*}
g_{\underline{\omega}}(X,Y) \, \mu = {\frac{1}{6}} \epsilon^{ijk} (X \hook \omega_i) \wedge (Y \hook \omega_j) \wedge \omega_k,
  \end{equation*}
  where $\mu$ is a positive top-form such that
  \begin{equation*}
\frac{1}{2} \omega_i \wedge \omega_j = Q_{ij} \mu, \quad \det Q_{ij} = 1.
  \end{equation*}
  The hypersymplectic flow is defined~\cite{FineYao18} by
  \begin{equation} \label{eq:hyperflow}
\frac{d \, \underline{\omega}}{dt}  = d \left( Q d^* (Q^{-1} \underline{\omega}) \right).
  \end{equation}

In the context of $G_2$-lifts: suppose we have a solution $\underline{\omega}(t)$ to the hypersymplectic flow. Then if we let $M^7 = T^3 \times X^4$ and define
  \begin{equation}
\varphi(t) = - dr^1 \wedge dr^2 \wedge dr^3 + \sum_i dr^i \wedge \omega_i(t), \quad \varphi(t) \in \Omega^3(M^7)
\end{equation}
where $r^1$ $r^2$, and $r^3$ denote the angle coordinates on the $T^3$, then the evolving family $(M^7, \varphi(t))$ is a solution to the $G_2$-Laplacian flow
  \begin{equation} \label{eq:G2Laplacianflow}
\frac{d \varphi}{dt} = d d^* \varphi, \quad d \varphi(0) = 0.
  \end{equation}
  Indeed, this property is how Fine--Yao~\cite{FineYao18} derived the hypersymplectic flow.

  Test examples can be found by solving parabolic complex Monge--Amp\`ere equations; it was noticed in~\cite{FineHeYao25} that the calculation of~\cite{PS} can be interpreted as producing special solutions to the hypersymplectic flow with K\"ahler initial data.

  \begin{propn}~\cite{PS}
    Let $X^4$ be a compact complex twofold admitting a holomorphic volume form $\Omega$ and a K\"ahler metric $\omega_0$ such that $d \omega_0=0$. Start the ${\rm MA}^{1/3}$ flow
      \[
\frac{d u}{dt} =\bigg( e^{-a(x)} \frac{\det (\omega_0 + i \partial \bar{\partial} u)}{\det \omega_0}  \bigg)^{1/3}.
\]
with $a(x) = 2 \log |\Omega|_{\omega_0} - \log 864$. This flow exists for all time. Denote $\tilde{\omega}_t = \omega_0 + i \partial \bar{\partial} u(t)$. Denote $\Theta_t$ a family of diffeomorphisms determined by $\Theta_0 = id$ and ${\frac{d}{dt}} \Theta_t = \nabla_{\tilde{g}_t} \log |\Omega|_{\tilde{\omega}_t}$.

 Then if we let
   \[
\omega_t = \Theta_t^* \tilde{\omega}_t, \quad \Omega_t = \Theta_t^* \Omega_0,
\]
the hypersymplectic triple
\[
\underline{\omega}(t) = (\omega_t, {\rm Re} \, \Omega_t, {\rm Im} \, \Omega_t).
\]
solves the hypersymplectic flow~\eqref{eq:hyperflow}. Equivalently,
  \begin{equation}
\varphi(t) = - dr^1 \wedge dr^2 \wedge dr^3 + \sum_i dr^i \wedge \omega_i(t), \quad \varphi(t) \in \Omega^3(M^7)
\end{equation}
is an evolving family of $G_2$ structures on $M^7= T^3 \times X^4$ solving the $G_2$-Laplacian flow~\eqref{eq:G2Laplacianflow}.
\end{propn}

\section{Curvature of conformally coclosed \texorpdfstring{$G_2$}{G2}-structures} \label{sect-curvature}

Having introduced flows of conformally coclosed $G_2$-structures through their connection with familiar flows in complex geometry, we now turn to deriving various identities for conformally coclosed structures. The main objective of this section is to derive the expression for the Ricci curvature of a conformally coclosed $G_2$-structure. We use the definitions, notation, and identities from~\cite{Kar09, DGK25} throughout.

\subsection{Preliminary identities for conformally coclosed \texorpdfstring{$G_2$}{G2}-structures} \label{sec:cc-prelim}

Certain tensors associated to torsion and curvature which feature in this article are:
\begin{align*}
    (\sserif{P}T)_{ij} &= T_{kl} \psi_{klij}, \\
    (\sserif{V}T)_k &= T_{pq} \varphi_{pqk}, \\
    \tensor[_3]{K}{_i_j} &= (\nabla_p T_{qi}) \varphi_{pqj}.
\end{align*}

Recall the important $G_2$-Bianchi identity relating the torsion and curvature of a $G_2$-manifold:
\begin{align} \label{eqn-G2-Bianchi-a}
    \nabla_i T_{jk} - \nabla_j T_{ik} = T_{ip}T_{jq} \varphi_{pqk} + \frac{1}{2} R_{ijpq} \varphi_{pqk}.
\end{align}

In~\cite[Theorem 5.8]{DGK25}, this identity was decomposed into its various representation-theoretic components, resulting in independent identities relating $\Rm$, $\nabla T$, and $T$. The relations which are relevant for our purposes in the present article are:
    \begin{alignat}{2}
        \label{eqn-G2-Bianchi-7a}
        &\quad \div T^t - \nabla (\tr T) - T (\sserif{V}T) = 0, \\
        \label{eqn-G2-Bianchi-7b}
        &\quad \inner{\nabla T}{\psi} - (\tr T) \sserif{V}T + \sserif{V}(T^2) + T^t (\sserif{V}T) = 0, \\
        \label{eqn-G2-Bianchi-14}
        &\quad \pi_{14}(\tensor[_3]{K}{}) = - (\tr T) T_{14} + (T^2)_{14} + \big( (\sserif{P}T) T \big){}_{14}.
    \end{alignat}
We will use these to derive useful relations in the specific situation when the $G_2$-structure $\varphi$ is conformally coclosed. We are particularly interested in~\eqref{eqn-G2-Bianchi-7a} and~\eqref{eqn-G2-Bianchi-14}. (It can be shown that in the conformally coclosed case,~\eqref{eqn-G2-Bianchi-7b} gives no information.)

Recall that in our setting, $d (e^{-2f} \psi) = 0$, which by~\eqref{eq:torsion-forms} implies that
\begin{align*}
    \tau_1 = \frac{1}{2} df \text{ and } \tau_2 = 0.
\end{align*}
In terms of the torsion tensor $T$, we then obtain from~\eqref{eq:torsion-relations} that
\begin{equation} \label{eq:torsion-ccc}
    T_{ij} = \frac{1}{4} \tau_0 g_{ij} - (\tau_3')_{ij} + \frac{1}{2} (\nabla_p f) \varphi_{pij}
\end{equation}
which in turn implies that
\begin{equation} \label{eq:ccc-relations}
\begin{aligned}
    (\sserif{P}T)_{ij} &= T_{kl} \psi_{klij} = -2 (\nabla_p f) \varphi_{pij}, \\
    (\sserif{V}T)_k &= T_{ij} \varphi_{ijk} = 3 (\nabla_k f).
\end{aligned}
\end{equation}

\begin{lem}
    \label{lem-conf-coclosed-T(VT)-divT}
    If $\varphi$ is a conformally coclosed $G_2$-structure with $d(e^{-2f} \psi) = 0$, then
    \begin{align*}
        (\nabla_p f) T_{kp} = (\nabla_p f) T_{pk} \quad \text{and} \quad \nabla_p T_{kp} = \nabla_p T_{pk}
    \end{align*}
    or, in terms of the notation in~\cite{DGK25},
    \begin{align*}
    T(\sserif{V}T) = T^t (\sserif{V}T) \quad \text{and} \quad \div T = \div T^t.
\end{align*}
\end{lem}
\begin{proof}
    From~\eqref{eq:torsion-ccc} we have
    \begin{align*}
        (\nabla_p f) (T_{kp} - T_{pk}) = (\nabla_p f) (\nabla_r f) \varphi_{rkp} = 0.
    \end{align*}
    From~\eqref{eq:ccc-relations} we get
    \begin{align} \label{eq:ccc-temp}
        (\nabla_r f) T_{pm} \psi_{mrkp} = 2 (\nabla_r f) (\nabla_s f) \varphi_{srk} = 0,    
    \end{align}
    and so
    \begin{align*}
        \nabla_p (T_{kp} - T_{pk}) &= \nabla_p \Big( (\nabla_r f) \varphi_{rkp} \Big) \\
        &= (\nabla_p \nabla_r f) \varphi_{rkp} + (\nabla_r f) T_{pm} \psi_{mrkp} = 0. \qedhere
    \end{align*}
\end{proof}

\begin{propn}
    \label{propn-G2-Bianchi-conf-coclosed-7}
    If $\varphi$ is a conformally coclosed $G_2$-structure with $d(e^{-2f} \psi) = 0$, then
    \begin{align}
        \label{eqn-G2-Bianchi-conf-coclosed-7}
        \frac{3}{2} (\nabla_k \tau_0) + \nabla_m (\tau_3')_{km} = - \frac{3}{4} \tau_0 (\nabla_k f) + 3 (\nabla_m f) (\tau_3')_{km}
    \end{align}
\end{propn}
\begin{proof}
    To show this, we write~\eqref{eqn-G2-Bianchi-7a} in a local orthonormal frame as
    \begin{align} \label{eq:ccc-temp2}
        \nabla_k T_{pp} = \nabla_p T_{kp} - T_{km} T_{pq} \varphi_{pqm}.
    \end{align}
    First, using~\eqref{eq:torsion-ccc} and~\eqref{eq:ccc-temp} we compute
     \begin{align*}
        \nabla_p T_{pk} & = \frac{1}{4} \nabla_k \tau_0 - \nabla_p (\tau_3')_{pk} + \frac{1}{2} (\nabla_p \nabla_r f) \varphi_{rpk} +
        (\nabla_r f) T_{pq} \psi_{qrpk} \\
        & = \frac{1}{4} \nabla_k \tau_0 - \nabla_p (\tau_3')_{pk}.
    \end{align*}
    From~\eqref{eq:ccc-relations} we have
    \begin{align*}
        T_{km} T_{pq} \varphi_{pqm} = \frac{3}{4} (\nabla_k f) \tau_0 - 3 (\nabla_m f) (\tau_3')_{km}.
    \end{align*}
    Substituting both of these into~\eqref{eq:ccc-temp2}, we obtain
    \begin{align*}
        \frac{7}{4} (\nabla_k \tau_0) = \frac{1}{4} (\nabla_k \tau_0) - \nabla_m (\tau_3')_{km} - \frac{3}{4} \tau_0 (\nabla_k f) + 3 (\nabla_m f) (\tau_3')_{km},
    \end{align*}
    which simplifies to~\eqref{eqn-G2-Bianchi-conf-coclosed-7}.
\end{proof}

The next result is a consequence of the identity~\eqref{eqn-G2-Bianchi-14}. We do not actually make use of it in this paper, but we present it here for for possible future applications to the study of conformally coclosed $G_2$-structures.
\begin{propn}
    \label{propn-G2-Bianchi-conf-coclosed-14}
    If $\varphi$ is a conformally coclosed $G_2$-structure with $d(e^{-2f} \psi) = 0$, then
    \begin{align}
        \label{eqn-G2-Bianchi-conf-coclosed-14}
        \frac{1}{6} \big( \nabla_p (\tau_3')_{pk} \big) \varphi_{kia} + \frac{1}{2} \Big[ \big( \nabla_p (\tau_3')_{iq} \big) \varphi_{pqa} - \big( \nabla_p (\tau_3')_{aq} \big) \varphi_{pqi} \Big] = 0.
    \end{align}
\end{propn}
\begin{proof}
    We will apply the conformally coclosed condition to the identity~\eqref{eqn-G2-Bianchi-14}. We sketch the ideas, leaving some of the explicit computation for the reader. Starting with the LHS, one computes that
    \begin{align*}
        (\tensor[_3]{K}{})_{ia} &= (\nabla_p T_{qi}) \varphi_{pqa} \\
        &= \nabla_p \Big[ \frac{1}{4} \tau_0 g_{qi} - (\tau_3')_{qi} + \frac{1}{2} (\nabla_r f) \varphi_{rqi} \Big] \varphi_{pqa} \\
        &= \frac{1}{4} (\nabla_k \tau_0) \varphi_{kia} - \big( \nabla_p (\tau_3')_{iq} \big) \varphi_{pqa} + \frac{1}{2} (\nabla_p \nabla_r f) \big[ g_{rp} g_{ia} - g_{ra} g_{ip} - \psi_{ripa} \big] + \frac{1}{2} (\nabla_r f) T_{pm} \psi_{mrqi} \varphi_{pqa} \\
        &= \frac{1}{4} (\nabla_k \tau_0) \varphi_{kia} - \big( \nabla_p (\tau_3')_{iq} \big) \varphi_{pqa} + \frac{1}{2} (\Delta f) g_{ia} - \frac{1}{2} (\nabla_i \nabla_a f) \\
        &\qquad + \frac{1}{2} (\nabla_r f) T_{pm} \big[ g_{am} \varphi_{pir} - g_{ai} \varphi_{pmr} + g_{ar} \varphi_{pmi} - g_{pm} \varphi_{air} + g_{pi} \varphi_{amr} - g_{pr} \varphi_{ami} \big] \\
        &= \frac{1}{4} (\nabla_k \tau_0) \varphi_{kia} - \frac{1}{2} \Big[ \big( \nabla_p (\tau_3')_{iq} \big) \varphi_{pqa} + \big( \nabla_p (\tau_3')_{aq} \big) \varphi_{pqi} \Big] \\
        &\qquad - \frac{1}{2} \Big[ \big( \nabla_p (\tau_3')_{iq} \big) \varphi_{pqa} - \big( \nabla_p (\tau_3')_{aq} \big) \varphi_{pqi} \Big] + \frac{1}{2} (\Delta f) g_{ia} - \frac{1}{2} (\nabla_i \nabla_a f) \\
        &\qquad + \frac{1}{2} \tau_0 (\nabla_k f) \varphi_{kia} + \frac{1}{2} (\nabla_p f) (\tau_3')_{pk} \varphi_{kia} + \frac{1}{2} \Big[ (\nabla_p f) (\tau_3')_{iq} \varphi_{pqa} - (\nabla_p f) (\tau_3')_{aq} \varphi_{pqi} \Big] \\
        &\qquad + (\nabla_i f) (\nabla_a f) - \|\nabla f\|^2 g_{ia}.
    \end{align*}
    One can then check, using $(\pi_{14} \beta)_{ij} = \frac{1}{6} (4 \beta_{ij} + \psi_{ijkl} \beta_{kl})$, that
    \begin{align*}
        \big( \pi_{14} (\tensor[_3]{K}{}) \big){}_{ia} &= - \frac{1}{6} \big( \nabla_p (\tau_3')_{pk} \big) \varphi_{kia} - \frac{1}{2} \Big[ \big( \nabla_p (\tau_3')_{iq} \big) \varphi_{pqa} - \big( \nabla_p (\tau_3')_{aq} \big) \varphi_{pqi} \Big] \\
        &\qquad + \frac{1}{6} (\nabla_p f) (\tau_3')_{pk} \varphi_{kia} + \frac{1}{2} \Big[ (\nabla_p f) (\tau_3')_{iq} \varphi_{pqa} - (\nabla_p f) (\tau_3')_{aq} \varphi_{pqi} \Big].
    \end{align*}

    Similar computations show that
    \begin{align*}
        \big( (T^2)_{14} \big){}_{ia} = - \frac{1}{6} (\nabla_p f) (\tau_3')_{pk} \varphi_{kia} - \frac{1}{2} \Big[ (\nabla_p f) (\tau_3')_{iq} \varphi_{pqa} - (\nabla_p f) (\tau_3')_{aq} \varphi_{pqi} \Big],
    \end{align*}
    and
    \begin{align*}
        \big( \big( (\sserif{P}T) T \big){}_{14} \big){}_{ia} = \frac{1}{3} (\nabla_p f) (\tau_3')_{pk} \varphi_{kia} + \Big[ (\nabla_p f) (\tau_3')_{iq} \varphi_{pqa} - (\nabla_p f) (\tau_3')_{aq} \varphi_{pqi} \Big].
    \end{align*}
     Combining the above expressions, the identity~\eqref{eqn-G2-Bianchi-14} implies~\eqref{eqn-G2-Bianchi-conf-coclosed-14} in the conformally coclosed setting.
\end{proof}

\subsection{The Ricci curvature of a conformally coclosed \texorpdfstring{$G_2$}{G2}-structure} \label{sec:Ricci}

By contracting~\eqref{eqn-G2-Bianchi-a} with $\varphi$, we can recover another form of the $G_2$-Bianchi identity
\begin{align*}
    \Ric_{ia} = (\nabla_p T_{iq} - \nabla_i T_{pq}) \varphi_{pqa} - T_{im} T_{ma} + (\tr T) T_{ia} + T_{pq} T_{ir} \psi_{apqr}.
\end{align*}
Using the fact that $\Ric$ is symmetric, we can also write a more symmetric version of this identity:
\begin{equation} \label{eqn-G2-Bianchi-b-symm}
\begin{aligned}
    \Ric_{ia} &= \frac{1}{2} \Big[ (\nabla_p T_{iq}) \varphi_{pqa} + (\nabla_p T_{aq}) \varphi_{pqi} \Big] - \frac{1}{2} \Big[ (\nabla_i T_{pq}) \varphi_{pqa} + (\nabla_a T_{pq}) \varphi_{pqi} \Big] \\
    &\qquad - \frac{1}{2} \Big[ T_{im} T_{ma} + T_{am} T_{mi} \Big] + \frac{1}{2} (\tr T) \Big[ T_{ia} + T_{ai} \Big] + \frac{1}{2} T_{pq} \Big[ T_{ir} \psi_{apqr} + T_{ar} \psi_{ipqr} \Big].
\end{aligned}
\end{equation}

From~\eqref{eqn-G2-Bianchi-b-symm}, we will now derive expressions for the Ricci curvature and the scalar curvature in the conformally coclosed case, in terms of $\tau_0$, $\tau_3'$, and $f$.

\begin{lem}
    If $\varphi$ is a conformally coclosed $G_2$-structure with $d(e^{-2f} \psi) = 0$, then
    \begin{align}
        \label{eqn-nabla-f-curl-T-conf-coclosed}
        (\nabla_p f) T_{iq} \varphi_{pqa} &= \frac{1}{4} \tau_0 (\nabla_k f) \varphi_{kia} - (\nabla_p f) (\tau_3')_{iq} \varphi_{pqa} + \frac{1}{2} (\nabla_i f) (\nabla_a f) - \frac{1}{2} \|\nabla f\|^2 g_{ia}
    \end{align}
    and
    \begin{align}
        \label{eqn-nabla-f-curl-T^t-conf-coclosed}
        (\nabla_p f) T_{qi} \varphi_{pqa} &= \frac{1}{4} \tau_0 (\nabla_k f) \varphi_{kia} - (\nabla_p f) (\tau_3')_{iq} \varphi_{pqa} - \frac{1}{2} (\nabla_i f) (\nabla_a f) + \frac{1}{2} \|\nabla f\|^2 g_{ia}.
    \end{align}
\end{lem}
\begin{proof}
These follow easily from~\eqref{eq:torsion-ccc}.
\end{proof}

\begin{propn}
    \label{propn-Ric-R-conf-coclosed}
    If $\varphi$ is a conformally coclosed $G_2$-structure with $d(e^{-2f} \psi) = 0$, then
    \begin{equation} \label{eqn-Ric-conf-coclosed}
    \begin{aligned}
        \Ric_{ia} &= - \frac{1}{2} \Big[ \big( \nabla_p (\tau_3')_{iq} \big) \varphi_{pqa} + \big( \nabla_p (\tau_3')_{aq} \big) \varphi_{pqi} \Big] - \frac{5}{2} (\nabla_i \nabla_a f) - \frac{1}{2} (\Delta f) g_{ia} \\
        &\qquad + \frac{3}{8} \tau_0^2 g_{ia} - \frac{5}{4} \tau_0 (\tau_3')_{ia} - (\tau_3')_{im} (\tau_3')_{ma} - \frac{5}{4} (\nabla_i f) (\nabla_a f) + \frac{5}{4} \|\nabla f\|^2 g_{ia}.
    \end{aligned}
    \end{equation}
    and
    \begin{align}
        \label{eqn-R-conf-coclosed}
        R = - 6 (\Delta f) + \frac{21}{8} \tau_0^2 - \|\tau_3'\|^2 + \frac{15}{2} \|\nabla f\|^2.
    \end{align}
\end{propn}
\begin{proof}
    This follows from direct computation. We do this term-by-term using~\eqref{eqn-G2-Bianchi-b-symm}. First, we compute
    \begin{align*}
        & \qquad \frac{1}{2} \Big[ (\nabla_p T_{iq}) \varphi_{pqa} + (\nabla_p T_{aq}) \varphi_{pqi} \Big] \\
        &= \frac{1}{2} \Big( \nabla_p \Big[ \frac{1}{4} \tau_0 g_{iq} - (\tau_3')_{iq} + \frac{1}{2} (\nabla_r f) \varphi_{riq} \Big] \Big) \varphi_{pqa} + (i \lrarr a) \\
        &= \frac{1}{8} (\nabla_p \tau_0) \varphi_{pia} - \frac{1}{2} \big( \nabla_p (\tau_3')_{iq} \big) \varphi_{pqa} + \frac{1}{4} (\nabla_p \nabla_r f) \big( g_{ra} g_{ip} - g_{rp} g_{ia} - \psi_{riap} \big) \\
        &\qquad + \frac{1}{4} (\nabla_r f) T_{pm} \psi_{mriq} \varphi_{pqa} + (i \lrarr a) \\
        &= - \frac{1}{2} \big( \nabla_p (\tau_3')_{iq} \big) \varphi_{pqa} + \frac{1}{4} (\nabla_i \nabla_a f) - \frac{1}{4} (\Delta f) g_{ia} \\
        &\qquad + \frac{1}{4} (\nabla_r f) T_{pm} \big( g_{am} \varphi_{pri} - g_{ar} \varphi_{pmi} + g_{ai} \varphi_{pmr} - g_{pm} \varphi_{ari} + g_{pr} \varphi_{ami} - g_{pi} \varphi_{amr} \big) + (i \lrarr a) \\
        &= - \frac{1}{2} \big( \nabla_p (\tau_3')_{iq} \big) \varphi_{pqa} + \frac{1}{4} (\nabla_i \nabla_a f) - \frac{1}{4} (\Delta f) g_{ia} - \frac{1}{4} (\nabla_p f) T_{qa} \varphi_{pqi} - \frac{1}{4} (\nabla_a f) T_{pq} \varphi_{pqi} \\
        &\qquad + \frac{1}{4} (\nabla_r f) T_{pq} \varphi_{pqr} g_{ia} + \frac{1}{4} (\nabla_p f) T_{iq} \varphi_{pqa} + (i \lrarr a)
     \end{align*}
     where the notation ${}+ (i \lrarr a)$ denotes that we are adding the same terms with the indices $i$ and $a$ swapped. Then substituting~\eqref{eqn-nabla-f-curl-T-conf-coclosed} and~\eqref{eqn-nabla-f-curl-T^t-conf-coclosed} into the above, some more computation yields that the first term is
          \begin{align*}
        \frac{1}{2} \Big[ (\nabla_p T_{iq}) \varphi_{pqa} + (\nabla_p T_{aq}) \varphi_{pqi} \Big] &= - \frac{1}{2} \Big[ \big( \nabla_p (\tau_3')_{iq} \big) \varphi_{pqa} + \big( \nabla_p (\tau_3')_{aq} \big) \varphi_{pqi} \Big] \\
        &\qquad + \frac{1}{2} (\nabla_i \nabla_a f) - \frac{1}{2} (\Delta f) g_{ia} - (\nabla_i f) (\nabla_a f) + \|\nabla f\|^2 g_{ia}.
    \end{align*}
    Next, again using~\eqref{eqn-nabla-f-curl-T-conf-coclosed}, we compute
    \begin{align*}
        -\frac{1}{2} \Big[ (\nabla_i T_{pq}) \varphi_{pqa} + (\nabla_a T_{pq}) \varphi_{pqi} \Big] &= - \frac{1}{2} \Big( \nabla_i \Big[ \frac{1}{2} (\nabla_r f) \varphi_{rpq} \Big] \Big) \varphi_{pqa} + (i \lrarr a) \\
        &= - \frac{3}{2} (\nabla_i \nabla_a f) - \frac{1}{4} (\nabla_r f) T_{im} \psi_{mrpq} \varphi_{pqa} + (i \lrarr a) \\
        &= - \frac{3}{2} (\nabla_i \nabla_a f) - (\nabla_p f) T_{iq} \varphi_{pqa} + (i \lrarr a) \\
        &= - 3 (\nabla_i \nabla_a f) + \Big[ (\nabla_p f) (\tau_3')_{iq} \varphi_{pqa} + (\nabla_p f) (\tau_3')_{aq} \varphi_{pqi} \Big] \\
        &\qquad - (\nabla_i f) (\nabla_a f) + \|\nabla f\|^2 g_{ia}.
    \end{align*}
    Thirdly, we have
    \begin{align*}
        - \frac{1}{2} \Big[ T_{im} T_{ma} + T_{am} T_{mi} \Big] &= - \frac{1}{2} \Big[ \frac{1}{4} \tau_0 g_{im} - (\tau_3')_{im} + \frac{1}{2} (\nabla_r f) \varphi_{rim} \Big] \Big[ \frac{1}{4} \tau_0 g_{ma} - (\tau_3')_{ma} + \frac{1}{2} (\nabla_s f) \varphi_{sma} \Big] \\
        &\qquad + (i \lrarr a) \\
        &= - \frac{1}{32} \tau_0^2 g_{ia} - \frac{1}{8} \tau_0 (\nabla_r f) \varphi_{ria} + \frac{1}{4} \tau_0 (\tau_3')_{ia} \\
        &\qquad - \frac{1}{2} (\tau_3')_{im} (\tau_3')_{ma} + \frac{1}{4} \Big[ (\nabla_p f) (\tau_3')_{iq} \varphi_{pqa} - (\nabla_p f) (\tau_3')_{aq} \varphi_{pqi} \Big] \\
        &\qquad - \frac{1}{8} (\nabla_r f) (\nabla_s f) \big( g_{ra} g_{is} - g_{rs} g_{ia} - \psi_{rias} \big) + (i \lrarr a) \\
        &= -\frac{1}{16} \tau_0^2 g_{ia} + \frac{1}{2} \tau_0 (\tau_3')_{ia} - (\tau_3')_{im} (\tau_3')_{ma} - \frac{1}{4} (\nabla_i f) (\nabla_a f) + \frac{1}{4} \|\nabla f\|^2 g_{ia}.
    \end{align*}
    The fourth term becomes
    \begin{align*}
        \frac{1}{2} (\tr T) \Big[ T_{ia} + T_{ai} \Big] &= \frac{7}{8} \tau_0 \Big[ \frac{1}{4} \tau_0 g_{ia} - (\tau_3')_{ia} + \frac{1}{2} (\nabla_r f) \varphi_{ria} \Big] + (i \lrarr a) \\
        &= \frac{7}{16} \tau_0^2 g_{ia} - \frac{7}{4} \tau_0 (\tau_3')_{ia}.
    \end{align*}
    Finally, again using ~\eqref{eqn-nabla-f-curl-T-conf-coclosed}, the fifth term is
    \begin{align*}
        \frac{1}{2} T_{pq} \Big[ T_{ir} \psi_{apqr} + T_{ar} \psi_{ipqr} \Big] &= \frac{1}{2} \Big[ \frac{1}{2} (\nabla_m f) \varphi_{mpq} \Big] T_{ir} \psi_{apqr} + (i \lrarr a) \\
        &= (\nabla_p f) T_{iq} \varphi_{pqa} + (i \lrarr a) \\
        &= - \Big[ (\nabla_p f) (\tau_3')_{iq} \varphi_{pqa} + (\nabla_p f) (\tau_3')_{aq} \varphi_{pqi} \Big] \\
        &\qquad + (\nabla_i f) (\nabla_a f) - \|\nabla f\|^2 g_{ia}.
    \end{align*}
    Summing the previous five equations yields~\eqref{eqn-Ric-conf-coclosed} and tracing this in turn gives~\eqref{eqn-R-conf-coclosed}.
\end{proof}

\begin{rmk}
    We note that all the results throughout Section~\ref{sect-curvature} hold for coclosed $G_2$-structures as well. In this case, we simply have that $f$ is constant and so $\nabla f \equiv 0$.
\end{rmk}

\section{Modified \texorpdfstring{$G_2$}{G2}-anomaly flows}
\label{sect-modified-G2-Anomaly-flows}

In Section~\ref{sect-lifting} we derived a flow of conformally coclosed $G_2$-structures by considering an ansatz of the form $M^7= S^1 \times X^6$ and moving between complex geometry on $X^6$ and $G_2$-geometry on $M^7$. We now turn to developing the general theory of this flow on a general $G_2$-manifold $M^7$. In this section, we examine the fixed points of the flow and determine the evolution of the metric.

\subsection{Notation and setup} \label{sect-notation-setup}
Fix a reference volume form $\vol_R$ on $M^7$. For all $G_2$-structures $\varphi$ on $M^7$, we associate a scalar function $f_{\varphi,R} \in C^\infty(M^7,\mathbb{R})$ by the formula
\begin{align}
\label{eqn-dilaton}
    e^{4f} = \frac{\vol_{\varphi}}{\vol_R}.
\end{align}
Here and from now on, we drop the subscript and simply write $f=f_{\varphi,R}$. The function $f$ is sometimes called the dilaton function.

Next, fix a constant $C \in \bb{R}$. For all $G_2$-structures $\varphi$, we associate a $3$-form $H_C \in \Omega^3(M^7)$ by
\begin{align*}
    H_C = -e^{2f} d^* (e^{-2f} \psi) - (C-2) \big( (\tr T) \varphi \big).
\end{align*}
Our seemingly mysterious choice to write the coefficient of $(\tr T) \varphi$ in $H_C$ as $-(C-2)$ is so that the coefficient of $e^{8f} (\div X) \varphi$ in equation~\eqref{eq:DP-final} below is precisely $C$. In particular for $C = \frac{4}{3}$, we recover the distinguished skew-symmetric torsion tensor $H$ from~\eqref{eq:H}.

With this notation, the dilaton coflow derived in Remark~\ref{lift-modifiedaf} takes the form
\begin{align}
\label{eqn-G2-Anomaly-flow}
    \del_t (e^{-2f} \psi) &= -dH_C,
\end{align}
with initial data satisfying $d  (e^{-2f_0} \psi_0) = 0$. As the right hand side of the above equation is exact, we see that this flow preserves the conformally coclosed condition $d(e^{-2f} \psi)=0$. We note that coclosed structures with $d \psi_0 =0$ can be used as initial data upon when choosing the reference volume as $\vol_R = \vol_{\psi_0}$, and once the flow starts the closedness condition will morph to $d(e^{-2f_{\varphi,\varphi_0}} \psi)=0$.

Since, as we saw in~\eqref{eq:torsion-ccc}, the torsion tensor $T$ of a coclosed structure is given by
\begin{align*}
    T = \frac{1}{4} \tau_0 g - \tau_3' + \frac{1}{2} (\nabla f) \hook \varphi, 
\end{align*}
we can expand the expression for $H_C$ and use~\eqref{eq:torsion-forms} and~\eqref{eq:27-hook} to get
\begin{align*}
\label{eqn-H_C}
    H_C &= -e^{2f} d^* (e^{-2f} \psi) - (C-2) \big( (\tr T) \varphi \big) \\
    &= -e^{2f} \star d (e^{-2f} \varphi) - \frac{7}{4} (C-2) \tau_0 \varphi \\
    &= 2 \star (df \wedge \varphi) - \star d\varphi - \frac{7}{4} (C-2) \tau_0 \varphi \\
    &= - 2 (\nabla f) \hook \psi - \star \Big( \tau_0 \psi + \frac{3}{2} df \wedge \varphi + \star \tau_3 \Big) - \frac{7}{4} (C-2) \tau_0 \varphi \\
    &= - \frac{7}{4} \Big( C - \frac{10}{7} \Big) \tau_0 \varphi - \tau_3 - \frac{1}{2} (\nabla f) \hook \psi \\
    &= - \Big[ \frac{7}{12} \Big( C - \frac{10}{7} \Big) \tau_0 g + \tau_3' - \frac{1}{6} (\nabla f) \hook \varphi \Big] \diamond \varphi. \numberthis
\end{align*}

\begin{rmk} \label{rmk-torsionconnection}
The definition of $H$ from~\eqref{eq:H} may seem mysterious, but it arises naturally in physics~\cite{dlOMS2018, AMP24} and it can be mathematically motivated as follows. If $\varphi$ is a $G_2$-structure with $\tau_2 = 0$ (called \emph{integrable} by some authors or \emph{$G_2$-structures with torsion} by others) then there exists a unique connection $\nabla^H$ which is compatible with $\varphi$ (and hence metric compatible) that has totally skew-symmetric torsion $H$, which is given by
    \begin{align} \label{eq:first-H}
        H = \frac{1}{6} \tau_0 \varphi - \tau_3 - (\tau_1)^\sharp \hook \psi.
    \end{align}
(See~\cite[Theorem 4.8]{FI02} for more details.) If $\varphi$ is conformally coclosed with $d (e^{-2 f} \psi) = 0$, then $H$ as above corresponds precisely to~\eqref{eq:H}, using $\tau_1 = \frac{1}{2}f$ and~\eqref{eq:27-hook}. It also corresponds to $H_C$ in~\eqref{eqn-H_C} when $C = \frac{4}{3}$.
\end{rmk}

\begin{rmk}
    The introduction of the parameter $C$ allows us to study an entire family of flows similar to others already found in the literature:
    \begin{itemize} \setlength\itemsep{-1mm}
        \item setting $C = 0$ gives a conformally coclosed version of Grigorian's modified Laplacian coflow~\cite{Gri13} (where Grigorian's additional parameter $A$ has been set to zero);
        \item setting $C = \frac{4}{3}$ recovers the case distinguished in~\cite{AMP24}. In particular, fixed points of this flow would be special cases of integrable ($\tau_2 = 0$) $G_2$-structures with closed torsion $dH = 0$ (these are also known as strong $G_2$-structures with torsion)~\cite{FF25,FMMR24,IS23};
        \item setting $C = 2$ yields a conformally coclosed version of the Laplacian coflow~\cite{KMP12}. 
    \end{itemize}
\end{rmk}

We adopt the nomenclature of referring to the $C = \frac{4}{3}$ case as the $G_2$-anomaly flow with $\alpha'=0$, as this choice of $C$ is consistent with the equations of string theory~\cite{AMP24}. In line with~\cite{Gri13}, we refer to the the general $C$ case as \emph{modified} $G_2$-anomaly flows.

\subsection{Fixed points of the flow}
\label{subsect-fixed-points}

From~\eqref{eqn-G2-Anomaly-flow} and~\eqref{eqn-H_C}, we see that fixed points of the modified $G_2$-anomaly flow are conformally coclosed $G_2$-structures that satisfy
\begin{align*}
    -dH_C = d (B \diamond \varphi) = 0,
\end{align*}
where $B$ is the tensor
\begin{align}
\label{eqn-B}
    B_{ia} = \Big[ \frac{7}{12} \Big( C - \frac{10}{7} \Big) \tau_0 g_{ia} + (\tau_3')_{ia} - \frac{1}{6} (\nabla_p f) \varphi_{pia} \Big].
\end{align}
By writing
$$ d (B \diamond \varphi) = S \diamond \psi $$
and projecting this condition onto its $1$-, $7$-, and $27$-components (see Appendix~\ref{app-decomp-exact-4-forms}), we can obtain further properties of these fixed points.

Recall that the torsion tensor $T$ in this setting is given by
\begin{align} \label{eq:T-v2}
    T_{ia} = \frac{1}{4} \tau_0 g_{ia} - (\tau_3')_{ia} + \frac{1}{2} (\nabla_p f) \varphi_{pia}.
\end{align}

In order to project $S$ onto the $1$-component, we compute from~\eqref{eqn-B} and~\eqref{eq:T-v2} that
\begin{equation} \label{eq:trS-temp}
\begin{aligned}
    -\frac{1}{2} (\nabla_m B_{pq}) \varphi_{mpq} &= \frac{1}{2} (\Delta f) - \|\nabla f\|^2, \\
    \frac{1}{2} B_{mm} T_{kk} &= \frac{343}{96} \Big( C - \frac{10}{7} \Big) \tau_0^2, \\
    - \frac{1}{2} B_{pq} T_{qp} &= - \frac{49}{96} \Big( C - \frac{10}{7} \Big) \tau_0^2 + \frac{1}{2} \|\tau_3'\|^2 - \frac{1}{4} \|\nabla f\|^2, \\
    \frac{1}{4} B_{pq} T_{rs} \psi_{pqrs} &= \frac{1}{2} \|\nabla f\|^2.
\end{aligned}
\end{equation}
Substituting the above expressions into equation~\eqref{eqn-decomp-exact-4-form-1}, the fixed point condition $\tr S = 0$ becomes
\begin{align}
\label{eqn-fixed-pt-G2-Anomaly-flow-1}
    \tr S = \frac{1}{2} (\Delta f) - \frac{3}{4} \|\nabla f\|^2 + \frac{1}{2} \|\tau_3'\|^2 + \frac{49}{16} \Big( C - \frac{10}{7} \Big) \tau_0^2 = 0.
\end{align}

If $C > \frac{10}{7}$, then we have
\begin{align*}
    \Delta f \leq \frac{3}{2} \|\nabla f\|^2
\end{align*}
so by the maximum principle, $f$ must be constant. (This follows from the fact that $e^{-\frac{3}{2} f}$ is subharmonic.) It further follows in this case from~\eqref{eqn-fixed-pt-G2-Anomaly-flow-1} that both $\tau_3'$ and $\tau_0$ must vanish and $\varphi$ must be torsion-free.

(We remark that if $C = \frac{10}{7}$, the above analysis still holds, except that $\tau_0$ need not vanish, and so $\varphi$ may be torsion-free or only nearly parallel, because $\tau_0$ is then constant by~\eqref{eqn-G2-Bianchi-conf-coclosed-7}.)

To consider the vanishing of the $7$-component $S_7$ of $S$, one first computes using~\eqref{eqn-B} and~\eqref{eq:T-v2} that
\begin{align*}
    - \frac{1}{6} \nabla_m B_{km} &= - \frac{7}{72} \Big( C - \frac{10}{7} \Big) \nabla_k \tau_0 - \frac{1}{6} \nabla_m (\tau_3')_{km}, \\
    \frac{1}{6} (\nabla_k B_{mm}) &= \frac{49}{72} \Big( C - \frac{10}{7} \Big) \nabla_k \tau_0, \\
    - \frac{1}{12} (\nabla_m B_{pq}) \psi_{mpqk} &= \frac{1}{12} \tau_0 (\nabla_k f) + \frac{1}{18} (\nabla_m f) (\tau_3')_{km},
\end{align*}
\begin{align*}
    \frac{1}{12} B_{mm} T_{pq} \varphi_{pqk} &= \frac{49}{48} \Big( C - \frac{10}{7} \Big) \tau_0 (\nabla_k f), \\
    \frac{1}{12} T_{mm} B_{pq} \varphi_{pqk} &= - \frac{7}{48} \tau_0 (\nabla_k f), \\
    - \frac{1}{12} B_{pm} T_{mq} \varphi_{pqk} &= - \frac{7}{48} \Big( C - \frac{10}{7} \Big) \tau_0 (\nabla_k f) + \frac{1}{48} \tau_0 (\nabla_k f) + \frac{1}{18} (\nabla_m f) (\tau_3')_{km} \\
    - \frac{1}{12} T_{pm} B_{mq} \varphi_{pqk} &= - \frac{7}{48} \Big( C - \frac{10}{7} \Big) \tau_0 (\nabla_k f) + \frac{1}{48} \tau_0 (\nabla_k f) + \frac{1}{18} (\nabla_m f) (\tau_3')_{km} \\
    \frac{1}{12} B_{pq} \varphi_{pqm} T_{km} &= -\frac{1}{48} \tau_0 (\nabla_k f) + \frac{1}{12} (\nabla_m f) (\tau_3')_{km}, \\
    \frac{1}{12} T_{pq} \varphi_{pqm} B_{km} &= \frac{7}{48} \Big( C - \frac{10}{7} \Big) \tau_0 (\nabla_k f) + \frac{1}{4} (\nabla_m f) (\tau_3')_{km}.
\end{align*}

Substituting the above expressions into equation~\eqref{eqn-decomp-exact-4-form-7}, and using the fact that $X \mapsto X \hook \varphi$ is an isomorphism, some computation yields that the fixed point condition $S_7 = 0$ becomes
\begin{align*}
     \frac{1}{6} (S_7)_{ia} \varphi_{iak} = \Big( \frac{7}{12} C - \frac{5}{6} \Big) (\nabla_k \tau_0) + \Big( \frac{7}{8} C - \frac{31}{24} \Big) \tau_0 (\nabla_k f) + \frac{1}{6} \Big( 3 (\nabla_m f) (\tau_3')_{mk} - \nabla_m (\tau_3')_{km} \Big) = 0.
\end{align*}
Further substituting above the consequence of the $G_2$-Bianchi identity~\eqref{eqn-G2-Bianchi-conf-coclosed-7} gives
\begin{align}
\label{eqn-fixed-pt-G2-Anomaly-flow-7}
    \frac{1}{6} (S_7)_{ia} \varphi_{iak} = \frac{7}{12} (C - 1) (\nabla_k \tau_0) + \frac{7}{8} \Big( C - \frac{4}{3} \Big) \tau_0 (\nabla_k f) = 0.
\end{align}
From this, we see that if $C = \frac{4}{3}$, then $\tau_0$ must be constant.

We note that the condition~\eqref{eqn-fixed-pt-G2-Anomaly-flow-7} can also be rewritten as
\begin{align}
\label{eqn-fixed-pt-7}
    \tau_0 = K \exp \Big( - \frac{3 (C - \frac{4}{3})}{2 (C-1)} f \Big)
\end{align}
for some constant $K$ provided that $C \neq 1$.

Finally, we can compute the $27$-component $S_{27}$ of $S$. First, one computes that
\begin{align*}
    - \frac{1}{4} \big( (\nabla_i B_{pq}) \varphi_{pqa} + (\nabla_a B_{pq}) \varphi_{pqi} \big) &= \frac{1}{2} (\nabla_i \nabla_a f) + \frac{1}{6} (\nabla_i f) (\nabla_a f) - \frac{1}{6} \|\nabla f\|^2 g_{ia} \\
    &\qquad - \frac{1}{6} \big( (\nabla_p f) (\tau_3')_{iq} \varphi_{pqa} + (\nabla_p f) (\tau_3')_{aq} \varphi_{pqi} \big)
\end{align*}
and
\begin{align*}
    \frac{1}{4} \Big( \big( \nabla_p (B_{iq} + B_{qi}) \big) \varphi_{pqa} + \big( \nabla_p (B_{aq} + B_{qa}) \big) \varphi_{pqi} \Big) = \frac{1}{2} \Big( \big( \nabla_p (\tau_3')_{iq} \big) \varphi_{pqa} + \nabla_p (\tau_3')_{aq} \big) \varphi_{pqi} \Big).
\end{align*}

More computations yield
\begin{align*}
    \frac{1}{4} B_{mm} \big( T_{ia} + T_{ai} \big) &= \frac{49}{96} \Big( C - \frac{10}{7} \Big) \tau_0^2 g_{ia} - \frac{49}{24} \Big( C - \frac{10}{7} \Big) \tau_0 (\tau_3')_{ia}, \\
    - \frac{1}{4} \big( B_{im} T_{ma} + B_{am} T_{mi} \big) &= \frac{1}{24} (\nabla_i f) (\nabla_a f) - \frac{1}{24} \|\nabla f\|^2 g_{ia} \\
    &\qquad - \frac{7}{96} \Big( C - \frac{10}{7} \Big) \tau_0^2 g_{ia} + \frac{1}{2} (\tau_3')_{im} (\tau_3')_{ma} \\
    &\qquad - \frac{1}{12} \big( (\nabla_p f) (\tau_3')_{iq} \varphi_{pqa} + (\nabla_p f) (\tau_3')_{aq} \varphi_{pqi} \big) \\
    &\qquad + \frac{7}{24} \Big( C - \frac{10}{7} \Big) \tau_0 (\tau_3')_{ia} - \frac{1}{8} \tau_0 (\tau_3')_{ia},
\end{align*}
\begin{align*}
    \frac{1}{4} T_{mm} \big( B_{ia} + B_{ai} \big) &= \frac{49}{96} \Big( C - \frac{10}{7} \Big) \tau_0^2 g_{ia} + \frac{7}{8} \tau_0 (\tau_3')_{ia}, \\
    - \frac{1}{4} \big( T_{im} B_{ma} + T_{am} B_{mi} \big) &= \frac{1}{24} (\nabla_i f) (\nabla_a f) - \frac{1}{24} \|\nabla f\|^2 g_{ia} \\
    &\qquad - \frac{7}{96} \Big( C - \frac{10}{7} \Big) \tau_0^2 g_{ia} + \frac{1}{2} (\tau_3')_{im} (\tau_3')_{ma} \\
    &\qquad + \frac{1}{12} \big( (\nabla_p f) (\tau_3')_{iq} \varphi_{pqa} + (\nabla_p f) (\tau_3')_{aq} \varphi_{pqi} \big) \\
    &\qquad + \frac{7}{24} \Big( C - \frac{10}{7} \Big) \tau_0 (\tau_3')_{ia} - \frac{1}{8} \tau_0 (\tau_3')_{ia},
\end{align*}
and
\begin{align*}
    \frac{1}{4} B_{pq} \big( \psi_{pqim} T_{ma} + \psi_{pqam} T_{mi} \big) &= \frac{1}{6} (\nabla_i f) (\nabla_a f) - \frac{1}{6} \|\nabla f\|^2 g_{ia} \\
    &\qquad + \frac{1}{6} \big( (\nabla_p f) (\tau_3')_{iq} \varphi_{pqa} + (\nabla_p f) (\tau_3')_{aq} \varphi_{pqi} \big), \\
    \frac{1}{4} T_{pq} \big( \psi_{pqim} B_{ma} + \psi_{pqam} B_{mi} \big) &= \frac{1}{6} (\nabla_i f) (\nabla_a f) - \frac{1}{6} \|\nabla f\|^2 g_{ia} \\
    &\qquad + \frac{1}{2} \big( (\nabla_p f) (\tau_3')_{iq} \varphi_{pqa} + (\nabla_p f) (\tau_3')_{aq} \varphi_{pqi} \big), \\
        - \frac{1}{4} B_{pq} T_{rs} \big( \varphi_{pri} \varphi_{qsa} + \varphi_{pra} \varphi_{qsi} \big) &= \frac{5}{12} (\nabla_i f)(\nabla_a f) - \frac{1}{6} \|\nabla f\|^2 g_{ia} \\
    &\qquad - \frac{7}{16} \Big( C - \frac{10}{7} \Big) \tau_0^2 g_{ia} + \frac{1}{2} (\tau_3')_{pq} (\tau_3')_{rs} \varphi_{pri} \varphi_{qsa} \\
    &\qquad - \frac{7}{24} \Big( C - \frac{10}{7} \Big) \tau_0 (\tau_3')_{ia} + \frac{1}{8} \tau_0 (\tau_3')_{ia}.
\end{align*}

Substituting all of the above expressions, as well as those from~\eqref{eq:trS-temp}, into equation~\eqref{eqn-decomp-exact-4-form-27} and combining terms, the fixed point condition $S_{27} = 0$ becomes
\begin{align*}
\label{eqn-fixed-pt-G2-Anomaly-flow-27}
    S_{27} &= \frac{1}{2} (\nabla_i \nabla_a f) - \frac{1}{14} (\Delta f) g_{ia} + \frac{1}{2} \Big( \big(\nabla_p (\tau_3')_{iq} \big) \varphi_{pqa} + \big( \nabla_p (\tau_3')_{aq} \big) \varphi_{pqi} \Big) \\
    &\qquad + (\nabla_i f) (\nabla_a f) - \frac{1}{7} \|\nabla f\|^2 g_{ia} \\
    &\qquad + (\tau_3')_{im} (\tau_3')_{ma} + \frac{1}{2} (\tau_3')_{pq} (\tau_3')_{rs} \varphi_{pri} \varphi_{qsa} - \frac{1}{14} \|\tau_3'\|^2 g_{ia} \\
    &\qquad + \frac{1}{2} \big( (\nabla_p f) (\tau_3')_{iq} \varphi_{pqa} + (\nabla_p f) (\tau_3')_{aq} \varphi_{pqi} \big) - \frac{7}{4} \Big( C - \frac{13}{7} \Big) \tau_0 (\tau_3')_{ia} = 0. \numberthis
\end{align*}

In summary, in this section we have established the following:
\begin{thm}
\label{thm-fixed-pts-G2-Anomaly-flow}
    The fixed points $\varphi$ of the modified $G_2$-anomaly flow
    \begin{align}
        \del_t (e^{-2f} \psi) = -dH_C = d \big( e^{2f} d^* (e^{-2f} \psi) \big) + (C-2) d \big( (\tr T) \varphi \big). \tag{\ref{eqn-G2-Anomaly-flow}}
    \end{align}
    satisfy the following:
    \begin{itemize} \setlength\itemsep{-1mm}
        \item if $C \neq 1$, then
        \begin{align*}
            \tau_0 = K \exp \Big( - \frac{3 (C - \frac{4}{3})}{2 (C-1)} f \Big)
        \end{align*}
        for some constant $K$;

        \item if $C = \frac{4}{3}$, then $\tau_0$ is constant;

        \item if $C = \frac{10}{7}$, then $f$ is constant and $\varphi$ is nearly parallel or torsion-free;

        \item if $C > \frac{10}{7}$, then $f$ is constant and $\varphi$ is torsion-free.
    \end{itemize}
\end{thm}

\begin{rmk} \label{rmk:fixed-points}
  Theorem~\ref{thm-fixed-pts-G2-Anomaly-flow} establishes most of Theorem~\ref{thm:fixed-points-intro}, except that $(g,H,f)$ is a generalized Ricci soliton when $C=\frac{4}{3}$. This will be shown in Corollary~\ref{susy2eom} below.
\end{rmk}

We note that we can recover analogous results~\cite{FF25} for the (modified) Laplacian coflow
\begin{align}
\label{eqn-Laplacian-coflow}
    \del_t \psi =  d d^* \psi + (C-2) d \big( (\tr T) \varphi \big).
\end{align}
for coclosed $G_2$-structures by setting $f \equiv 0$ throughout. In particular, the $1$-component condition for a fixed point~\eqref{eqn-fixed-pt-G2-Anomaly-flow-1} becomes
\begin{align}
\label{eqn-fixed-pt-Laplacian-coflow-1}
    \frac{1}{2} \|\tau_3'\|^2 + \frac{49}{16} \Big( C - \frac{10}{7} \Big) \tau_0^2 = 0
\end{align}
and the $7$-component condition~\eqref{eqn-fixed-pt-G2-Anomaly-flow-7} for a fixed point becomes
\begin{align}
\label{eqn-fixed-pt-Laplacian-coflow-7}
    \frac{7}{12} (C-1) (\nabla_k \tau_0) = 0.
\end{align}

If $C \neq 1$, then~\eqref{eqn-fixed-pt-Laplacian-coflow-7} says $\tau_0$ must be constant which in turn from~\eqref{eqn-fixed-pt-Laplacian-coflow-1} implies that $\tau_3'$ (and thus also $\tau_3$) must have constant norm. One can also check in this case that
\begin{align*}
    0 = d^2 \varphi = d (\tau_0 \psi + \star \tau_3) = d \star \tau_3 \implies d^* \tau_3 = 0.
\end{align*}
Moreover, from~\eqref{eq:T-v2} and~\eqref{eqn-G2-Bianchi-conf-coclosed-7} we obtain $\div T = 0$.

\begin{propn}[\cite{FF25} Proposition 8.3]
\label{propn-fixed-pts-Laplacian-coflow}
    The fixed points $\varphi$ of the modified Laplacian flow~\eqref{eqn-Laplacian-coflow} satisfy the following:
    \begin{itemize} \setlength\itemsep{-1mm}
        \item if $C \neq 1$, then $\tau_0$ is constant, $\div T = 0$, and $\tau_3$ has constant norm and is coclosed;
        
        \item if $C = \frac{10}{7}$, then $\varphi$ is nearly parallel or torsion-free;

        \item if $C > \frac{10}{7}$, then $\varphi$ is torsion-free.
    \end{itemize}
\end{propn}

\begin{rmk}
Although we make use of the explicit expression for $S_{27}$ from~\eqref{eqn-fixed-pt-G2-Anomaly-flow-27} in the next section to compute the evolution of the metric, we were not able to use the vanishing of $S_{27}$ to obtain further useful information about the fixed points of the modified $G_2$-Anomaly flow. It is not clear how to do this.
\end{rmk}

\subsection{Evolution of the metric}
\label{sec-evol-metric}

In this subsection, we compute the flow of the metric $g$ induced by the modified $G_2$-anomaly flow~\eqref{eqn-G2-Anomaly-flow}. For general $C$ we find a Ricci flow with extra torsion terms, but when when $C=\frac{4}{3}$ many terms cancel and we see that the metric flows exactly by the bosonic Einstein equation~\eqref{GRF}, which is reminiscent of the generalized Ricci flow~\cite{GFS21}.

\begin{rmk}
We also note that our analysis resolves a question of~\cite{AMP24}, since we prove that the metric does indeed evolve by~\eqref{GRF} even though the supersymmetric condition $\tau_0 = $ \emph{constant} is not imposed along the flow (though we have seen earlier that $\tau_0 =$ \emph{constant} holds at a fixed point).
\end{rmk}

Suppose that the $G_2$-structure $\varphi$ evolves by
\begin{align*}
    \del_t \psi = A \diamond \psi = \Big( \frac{1}{7} (\tr A) g + A_7 + A_{27} \Big) \diamond \psi.
\end{align*}
This then induces~\cite{DGK25,Kar09} the flow
\begin{align*}
    \del_t \vol = (\tr A) \vol
\end{align*}
of the volume form and the flow
\begin{align*}
    \del_t g = \frac{2}{7} (\tr A) g + 2 A_{27}
\end{align*}
of the metric.

Recall that the dilaton $f$ was defined with respect to a reference volume form $\vol_R$. In particular,~\eqref{eqn-dilaton} then shows that
\begin{align} \label{eq:f-evol}
    4 e^{4f} (\del_t f) = \del_t (e^{4f}) = \del_t \Big( \frac{\vol}{\vol_R} \Big) = (\tr A) \frac{\vol}{\vol_R} = (\tr A) e^{4f} \implies \del_t f = \frac{1}{4} (\tr A).
\end{align}
Thus we have
\begin{align*}
    \del_t (e^{-2f} \psi) &= -2 e^{-2f} (\del_t f) \psi + e^{-2f} (\del_t \psi) \\
    &= - \frac{1}{2} e^{-2f} (\tr A) \psi + e^{-2f} \Big( \frac{1}{7} (\tr A) g + A_7 + A_{27} \Big) \diamond \psi \\
    &= e^{-2f} \Big( \frac{1}{56} (\tr A) g + A_7 + A_{27} \Big) \diamond \psi.
\end{align*}
Matching terms with
\begin{align*}
    \del_t (e^{-2f} \psi) = \Big( \frac{1}{7} (\tr S) + S_7 + S_{27} \Big) \diamond \psi,
\end{align*}
we see that
\begin{align*}
    \del_t g = \frac{2}{7} (\tr A) g + 2 A_{27} = e^{2f} \Big( \frac{16}{7} (\tr S) g + 2 S_{27} \Big),
\end{align*}
where $\tr S$ and $S_{27}$ were computed and set to $0$ in the previous subsection for the fixed point conditions~\eqref{eqn-fixed-pt-G2-Anomaly-flow-1} and~\eqref{eqn-fixed-pt-G2-Anomaly-flow-27}.

In particular, we have
\begin{align*}
    \del_t g_{ia} &= e^{2f} \Big[ (\nabla_i \nabla_a f) + (\Delta f) g_{ia} + \Big( \big( \nabla_p (\tau_3')_{iq} \big) \varphi_{pqa} + \big( \nabla_p (\tau_3')_{aq} \big) \varphi_{pqi} \Big) \\
    &\qquad \qquad + 2 (\nabla_i f) (\nabla_a f) - 2 \|\nabla f\|^2 g_{ia} + 7 \Big( C - \frac{10}{7} \Big) \tau_0^2 g_{ia} \\
    &\qquad \qquad + 2 (\tau_3')_{im} (\tau_3')_{ma} + (\tau_3')_{pq} (\tau_3')_{rs} \varphi_{pri} \varphi_{qsa} + \|\tau_3'\|^2 g_{ia} \\
    &\qquad \qquad + \big( (\nabla_p f) (\tau_3')_{iq} \varphi_{pqa} + (\nabla_p f) (\tau_3')_{aq} \varphi_{pqi} \big) - \frac{7}{2} \Big( C - \frac{13}{7} \Big) \tau_0 (\tau_3')_{ia} \Big]. \numberthis \label{eq:metric-flow-explicit}
\end{align*}

We want to compare the right hand side of~\eqref{eq:metric-flow-explicit} to the symmetric $2$-tensor $H_C^2$, defined by $(H_C)^2_{ia} = (H_C)_{pqi} (H_C)_{pqa}$. For a general $2$-tensor $B$, one can compute that
\begin{align*}
    (B \diamond \varphi)^2_{ia} & = (B \diamond \varphi)_{pqi} (B \diamond \varphi)_{pqa} \\
    &= \big( B_{ik} \varphi_{kpq} - B_{pk} \varphi_{kiq} + B_{qk} \varphi_{kip} \big) \big( B_{al} \varphi_{lpq} - B_{pl} \varphi_{laq} + B_{ql} \varphi_{lap} \big) \\
    &= B_{ik} \varphi_{kpq} B_{al} \varphi_{lpq} - 2 B_{ik} \varphi_{kpq} B_{pl} \varphi_{laq} - 2 B_{pk} \varphi_{kiq} B_{al} \varphi_{lpq} + 2 \big( B_{pk} \varphi_{kiq} - B_{qk} \varphi_{kip} \big) B_{pl} \varphi_{laq} \\
    &= 2 B_{im} B_{am} + 2 B_{mm} \big( B_{ia} + B_{ai} \big) + 2 B_{pq} B_{pq} g_{ia} - 2 B_{mi} B_{ma} \\
    &\qquad + 2 B_{pq} \big( B_{im} \psi_{pqma} + B_{am} \psi_{pqmi} \big) + 2 B_{pq} B_{rs} \varphi_{psi} \varphi_{qra}.
\end{align*}

Substituting our expression for $B$ from~\eqref{eqn-B} into the various terms above, one computes that
\begin{align*}
    2 B_{im} B_{am} &= - \frac{1}{18} (\nabla_i f) (\nabla_a f) + \frac{1}{18} \|\nabla f\|^2 g_{ia} \\
    &\qquad + \frac{49}{72} \Big( C - \frac{10}{7} \Big)^2 \tau_0^2 g_{ia} + 2 (\tau_3')_{im} (\tau_3')_{ma} \\
    &\qquad + \frac{1}{3} \big( (\nabla_p f) (\tau_3')_{iq} \varphi_{pqa} + (\nabla_p f) (\tau_3')_{aq} \varphi_{pqi} \big) + \frac{7}{3} \Big( C - \frac{10}{7} \Big) \tau_0 (\tau_3')_{ia},
\end{align*}
\begin{align*}
    2 B_{mm} \big( B_{ia} + B_{ai} \big) &= \frac{343}{36} \Big( C - \frac{10}{7} \Big)^2 \tau_0^2 g_{ia} + \frac{49}{3} \Big( C - \frac{10}{7} \Big) \tau_0 (\tau_3')_{ia}, \\
    2 B_{pq} B_{pq} g_{ia} &= \frac{1}{3} \|\nabla f\|^2 g_{ia} + \frac{343}{72} \Big( C - \frac{10}{7} \Big)^2 \tau_0^2 g_{ia} + 2 \|\tau_3'\|^2 g_{ia}, \\
    - 2 B_{mi} B_{ma} &= \frac{1}{18} (\nabla_i f) (\nabla_a f) - \frac{1}{18} \|\nabla f\|^2 g_{ia} \\
    &\qquad - \frac{49}{72} \Big( C - \frac{10}{7} \Big)^2 \tau_0^2 g_{ia} - 2 (\tau_3')_{im} (\tau_3')_{ma} \\
    &\qquad + \frac{1}{3} \big( (\nabla_p f) (\tau_3')_{iq} \varphi_{pqa} + (\nabla_p f) (\tau_3')_{aq} \varphi_{pqi} \big) - \frac{7}{3} \Big( C - \frac{10}{7} \Big) \tau_0 (\tau_3')_{ia}, \\
    2 B_{pq} \big( B_{im} \psi_{pqma} + B_{am} \psi_{pqmi} \big) &= - \frac{4}{9} (\nabla_i f) (\nabla_a f) + \frac{4}{9} \|\nabla f\|^2 g_{ia} \\
    &\qquad + \frac{4}{3} \big( (\nabla_p f) (\tau_3')_{iq} \varphi_{pqa} + (\nabla_p f) (\tau_3')_{aq} \varphi_{pqi} \big), \\
    2 B_{pq} B_{rs} \varphi_{psi} \varphi_{qra} &= - \frac{5}{9} (\nabla_i f) (\nabla_a f) + \frac{2}{9} \|\nabla f\|^2 g_{ia} + \frac{49}{12} \Big( C - \frac{10}{7} \Big)^2 \tau_0^2 g_{ia} \\
    &\qquad + 2 (\tau_3')_{pq} (\tau_3')_{rs} \varphi_{pri} \varphi_{qsa} - \frac{7}{3} \Big( C - \frac{10}{7} \Big) \tau_0 (\tau_3')_{ia}.
\end{align*}

Summing up the above expressions, we obtain
\begin{align} \label{H-squared}
    (H_C)^2_{ia} &= - (\nabla_i f) (\nabla_a f) + \|\nabla f\|^2 g_{ia} + \frac{147}{8} \Big( C - \frac{10}{7} \Big)^2 \tau_0^2 g_{ia} + 2 (\tau_3')_{pq} (\tau_3')_{rs} \varphi_{pri} \varphi_{qsa} + 2 \|\tau_3'\|^2 g_{ia} \\
    &\qquad + 2 \big( (\nabla_p f) (\tau_3')_{iq} \varphi_{pqa} + (\nabla_p f) (\tau_3')_{aq} \varphi_{pqi} \big) + 14 \Big( C - \frac{10}{7} \Big) \tau_0 (\tau_3')_{ia}.
\end{align}

We can now combine the above with our expression~\eqref{eqn-Ric-conf-coclosed} for the Ricci tensor, which we reproduce here:
\begin{equation} \tag{\ref{eqn-Ric-conf-coclosed}}
\begin{aligned}
    \Ric_{ia} &= - \frac{5}{2} (\nabla_i \nabla_a f) - \frac{1}{2} (\Delta f) g_{ia} - \frac{1}{2} \Big[ \big( \nabla_p (\tau_3')_{iq} \big) \varphi_{pqa} + \big( \nabla_p (\tau_3')_{aq} \big) \varphi_{pqi} \Big] \\
    &\qquad + \frac{3}{8} \tau_0^2 g_{ia} - \frac{5}{4} \tau_0 (\tau_3')_{ia} - (\tau_3')_{im} (\tau_3')_{ma} - \frac{5}{4} (\nabla_i f) (\nabla_a f) + \frac{5}{4} \|\nabla f\|^2 g_{ia}.
\end{aligned}
\end{equation}
Recalling~\eqref{eq:metric-flow-explicit}, with some more computation we finally deduce the following result:

\begin{thm} \label{thm:metric-evol}
    Under the modified $G_2$-anomaly flow~\eqref{eqn-G2-Anomaly-flow} restricted to conformally coclosed $G_2$-structures, the metric $g$ evolves by
    \begin{align*}
        \del_t g_{ia} &= e^{2f} \Big[ - 2 \Ric_{ia} - 4 (\nabla_i \nabla_a f) + \frac{1}{2}(H_C)^2_{ia} \\
        &\qquad \qquad - \frac{147}{16} \Big( C - \frac{4}{3} \Big) \Big( C - \frac{16}{7} \Big) \tau_0^2 g_{ia} - \frac{21}{2} \Big( C - \frac{4}{3} \Big) \tau_0 (\tau_3')_{ia} \Big].
    \end{align*}
\end{thm}

In particular, we observe that the metric evolves by a conformal generalized Ricci flow plus Hessian plus lower-order terms~\cite{GFS21}. Moreover, in the distinguished $C = \frac{4}{3}$ case, the lower-order discrepancies vanish.

\subsection{Comments on heterotic \texorpdfstring{$G_2$}{G2}} \label{sec:heterotic}

Let us put this discussion in the context of heterotic string theory. The heterotic $G_2$ equations at order $\alpha'=0$ are for a pair $(\varphi, f)$ where $\varphi$ is a $G_2$-structure and $f$ is a scalar function satisfying
\begin{equation} \label{g2-heterotic}
       d H = 0, \qquad d (e^{-2f} \psi) = 0,
\end{equation}
where $H = -e^{2f} d^* (e^{-2f} \psi) + \frac{2}{3} (\tr T) \varphi$. We see that fixed points of the flow
\[
\partial_t (e^{-2f} \psi) = - dH, \quad d (e^{-2f} \psi) =0, \quad e^{4 f} = e^{4 f_0} \frac{{\rm vol}_\varphi}{{\rm vol}_0},
\]
are solutions of the heterotic $G_2$ equations at $\alpha'=0$. In our earlier notation, this corresponds to setting $C=\frac{4}{3}$.

\begin{rmk}
  There are several \emph{inequivalent} formulations of the heterotic $G_2$ equations. For example, see~\cite{daSilvaEtAl2024, dlOLMMS2020, dlOMS2018, AMP24,CGFT2022, LotaySaEarp2023} for various versions. Here we simply consider the equations at order $\alpha'=0$ and turn off the gauge bundle.
  \end{rmk}

As a corollary of our calculations, and in particular from~\eqref{eqn-fixed-pt-G2-Anomaly-flow-7}, we find:

\begin{cor} \label{cor:constant-tau}
Let $M^7$ be a compact manifold and suppose $(\varphi, f)$ solves the heterotic $G_2$ equation at order $\alpha'=0$, namely~\eqref{g2-heterotic}. Then $\tau_0$ is constant.
\end{cor}

\begin{rmk}
It has already been observed that on a compact manifold with $G_2$-structure $\varphi$ such that $\tau_2 = 0$ and $dH = 0$, that $\tau_0$ must be constant. See~\cite[Corollary 3.2]{FF25} or~\cite[Theorem 7.1]{IS23}.
\end{rmk}

The condition that $\tau_0$ is fixed to a constant is often included as an extra condition in the heterotic $G_2$ equations, but in the current setup this is automatic.

Next, we consider the constraints on the Riemannian geometry implied by~\eqref{g2-heterotic}. In the language of string theory, we check whether the supersymmetry constraints at order $\alpha'=0$ restricted to 7-dimensions imply the equations of motion at order $\alpha'=0$ restricted to 7-dimensions. It was noticed in~\cite{daSilvaEtAl2024} that this is not true; while the supersymmetry constraints on $G_2$-geometries imply the graviton equation of motion and the $H$-equation of motion, the dilaton equation is modified. We include the derivation here for completeness.

\begin{cor} \label{susy2eom}~\cite{daSilvaEtAl2024}
Let $M^7$ be a compact manifold and suppose $(\varphi, f)$ solves the heterotic $G_2$ equation at order $\alpha'=0$, namely~\eqref{g2-heterotic} with
\begin{equation} \label{eq:H-again}
H = -e^{2f} d^* (e^{-2f} \psi) + \frac{2}{3} (\tr T) \varphi.
\end{equation}
  Then $(g_\varphi, H, f)$ solves the Riemannian system:
\begin{align}
   {\rm Ric} + 2 \nabla^2 f - \frac{1}{4} H^2 &= 0, \label{einstein} \\
  d^* ( e^{-2f} H) &= 0, \label{H-eom} \\
  R  + 4 \Delta f - 4 |\nabla f|^2- \frac{1}{12} \|H\|^2 &= \bigg( \frac{7}{6} \tau_0 \bigg)^2, \label{dilaton}
\end{align}
together with the closedness condition $dH=0$. 
\end{cor}

\begin{proof}
  The graviton equation~\eqref{einstein} follows from Theorem~\ref{thm:metric-evol} upon setting $C=\frac{4}{3}$ and $\partial_t g =0$. The $H$-equation of motion~\eqref{H-eom} follows by applying $d^* (e^{-2f} \cdot)$ to both sides of~\eqref{eq:H-again} and using that $\tr T$ is constant by Corollary~\ref{cor:constant-tau}. It remains to derive the modified dilaton equation~\eqref{dilaton}; the true dilaton equation (e.g.~\cite{AMP24}) has a zero on the right-hand side but here we are blocked by $\tau_0$. Taking the trace of~\eqref{einstein} gives
\begin{equation} \label{tr-einstein}
  R + 2 \Delta f = \frac{1}{4} \| H \|^2.
  \end{equation}
  This is not the dilaton equation~\eqref{dilaton}, so we must continue to manipulate terms. Taking the tensor norm squared of~\eqref{eq:first-H} and using the various contraction identities yields
\begin{equation} \label{eq:norm-H-squared}
  \| H \|^2 = 6 \|\nabla f\|^2 + \frac{7}{6} \tau_0^2 + 12 \|\tau_3'\|^2.
\end{equation}
Combining~\eqref{tr-einstein} and~\eqref{eq:norm-H-squared} gives
\begin{equation} \label{wecancheck2}
R + 2 \Delta f - \frac{1}{12} \| H \|^2 =\frac{7}{36} \tau_0^2 + 2 \|\tau_3'\|^2 + \|\nabla f\|^2.
\end{equation}
Next,~\eqref{eqn-fixed-pt-G2-Anomaly-flow-1} is the following identity for the Laplacian of $f$:
\begin{equation} \label{wecancheck3}
  2 \Delta f = 3\|\nabla f\|^2 -  2 \|\tau_3'\|^2 + \frac{7}{6}\tau_0^2.
\end{equation}
Adding up~\eqref{wecancheck2} and~\eqref{wecancheck3}, we see that the $\tau_3'$ terms cancel, and thus
\[
R + 4 \Delta f- \frac{1}{12} \| H \|^2 = 4 \|\nabla f\|^2 + \frac{49}{36}\tau_0^2.
\]
This completes the proof.
\end{proof}

As mentioned in Remark~\ref{rmk:fixed-points}, from Corollary~\ref{susy2eom} we obtain the final statement of Theorem~\ref{thm:fixed-points-intro}, namely that $(g,H,f)$ is a generalized Ricci soliton when $C = \frac{4}{3}$, because $H_C = H$ when $C = \frac{4}{3}$.

\begin{rmk}
  We have only considered the relevant equations of string theory with $\alpha'=0$. The name ``anomaly flow''~\cite{PPZ1, PPZ2} refers to terms at level $\alpha'$ involving a connection on a vector bundle $E \rightarrow M^7$, but the definition of the $G_2$-flow with linear $\alpha'$-corrections included remains open. Some progress on the $G_2$-anomaly flow with $\alpha'$-corrections is developed in~\cite{AMP24}, although one outstanding issue is the definition of a natural flow of a connection to a $G_2$-instanton.
\end{rmk}

\begin{rmk}
It is a theme in string theory that setting supersymmetry variations to zero produces special solutions to the equations of motion. We mention here a few references on this topic:~\cite{MartelliSparks,MinasianPetriniSvanes2017,AMP24,McOPic}.
\end{rmk}

\section{Short-time existence of the flow}
\label{sect-ste}

We now consider the well-posedness of the nonlinear equation under consideration. It is useful to define an auxiliary coclosed structure $\wtilde{\varphi}$ given by
\begin{align*}
    \wtilde{\varphi} = e^{-\frac{3}{2} f} \varphi.
\end{align*}
This, in turn, induces an auxiliary metric, Hodge star, dual $4$-form, and volume form given by
\begin{align*}
    \wtilde{g} = e^{-f} g, \qquad \wtilde{\star} = e^{- \frac{(7-2k)}{2} f} \star \quad \text{on $\Omega^k$}, \qquad \wtilde{\psi} = e^{-2f} \psi, \qquad \wtilde{\vol} = e^{-\frac{7}{2} f} \vol.
\end{align*}
The relevant torsion forms~\cite[Section 2.12]{DGK25} are then
\begin{align*}
    \wtilde{\tau}_0 = e^{\frac{1}{2} f} \tau_0, \qquad \wtilde{\tau}_1 = 0, \qquad \wtilde{\tau}_2 = 0, \qquad \wtilde{\tau}_3 = e^{-f} \tau_3.
\end{align*}
We can also define an auxiliary dilaton $\wtilde{f} = \frac{1}{8} f$ such that
\begin{align}
\label{eqn-dilaton-aux}
    e^{4 \wtilde{f}} = e^{\frac{1}{2} f} = e^{(4-\frac{7}{2}) f} = e^{-\frac{7}{2} f} \frac{\vol}{\vol_R} = \frac{\wtilde{\vol}}{\wtilde{\vol}_R},
\end{align}
where $\wtilde{\vol}_R = \vol_R$ is the reference volume form.

We can rewrite the flow~\eqref{eqn-G2-Anomaly-flow} in terms of the auxiliary structures. Firstly, from~\eqref{eqn-H_C} we have
\begin{align*}
    H_C &= -e^{8 \wtilde{f}} \Big[ \frac{7}{4} \Big( C - \frac{10}{7} \Big) \wtilde{\tau}_0 \wtilde{\varphi} + \wtilde{\tau}_3 + 4 (\wtilde{\nabla} \wtilde{f}) \hook \wtilde{\psi} \Big].
\end{align*}
and so the evolution equation becomes
\begin{align}
\label{eqn-G2-Anomaly-flow-aux}
    \del_t \wtilde{\psi} = d \bigg( e^{8 \wtilde{f}} \Big[ \frac{7}{4} \Big( C - \frac{10}{7} \Big) \wtilde{\tau}_0 \wtilde{\varphi} + \wtilde{\tau}_3 + 4 (\wtilde{\nabla} \wtilde{f}) \hook \wtilde{\psi} \Big] \bigg).
\end{align}
In this section we show that this induced flow has short-time existence under the restriction to coclosed $G_2$-structures. We do this using the method of Bryant--Xu~\cite{BX11} and Grigorian~\cite{Gri13}, which involves the use of DeTurck's Trick and the Nash--Moser implicit function theorem.

As such, we define an operator
\begin{align*}
    P (\wtilde{\varphi}) = d \bigg( e^{8 \wtilde{f}} \Big[ \frac{7}{4} \Big( C - \frac{10}{7} \Big) \wtilde{\tau}_0 \wtilde{\varphi} + \wtilde{\tau}_3 + 4 (\wtilde{\nabla} \wtilde{f}) \hook \wtilde{\psi} \Big] \bigg),
\end{align*}
where as before, we require a fixed reference volume form $\wtilde{\vol}_R = \vol_R$ to define the dilaton $f$.

For ease of notation, we will drop the tildes and let $\varphi$ denote a coclosed $G_2$-structure (with related tensors and operators $g$, $\star$, $\psi$, and $\vol$.)

\subsection{Linearizing the evolution equation}
\label{subsect-linearizing-evol-eqn}

Suppose we linearize the operator
\begin{align} \label{eq:P-op}
    P (\varphi) = d \bigg( e^{8 f} \Big[ \frac{7}{4} \Big( C - \frac{10}{7} \Big) \tau_0 \varphi + \tau_3 + 4 (\nabla f) \hook \psi \Big] \bigg)
\end{align}
in the direction of a closed $4$-form $\chi$ defined by
\begin{align}
\label{eqn-chi}
    \del_t \psi = \chi := \frac{4}{3} \alpha \psi - X^\flat \wedge \varphi + A \diamond \psi = \Big( \frac{1}{3} \alpha g - \frac{1}{3} X \hook \varphi + A \Big) \diamond \psi
\end{align}
for a function $\alpha$, vector field $X$, and traceless symmetric $2$-tensor $A$, which corresponds to the variation of $3$-forms
\begin{align} \label{eqn-flow3form}
    \del_t \varphi = \alpha \varphi + X \hook \psi + A \diamond \varphi = \Big( \frac{1}{3} \alpha g - \frac{1}{3} X \hook \varphi + A \Big) \diamond \varphi.
\end{align}
(See~\cite[Remark 3.8]{DGK25} and our equation~\eqref{eq:27-hook} to understand the second equalities in~\eqref{eqn-chi} and~\eqref{eqn-flow3form}.) 

One can check that $\star \chi = \frac{4}{3} \alpha \varphi + X \hook \psi - A \diamond \varphi$. Since the closed condition on $\chi$ is equivalent to $-d^* (\star \chi) = 0$, we can write the condition $d \chi = 0$ in a local orthonormal frame as
\begin{align*}
    0 &= \big( -d^* (\star \chi) \big){}_{ia} \\
    &= \nabla_p \Big( \frac{4}{3} \alpha \varphi_{pia} + X_m \psi_{mpia} - \big( A_{pm} \varphi_{mia} - A_{im} \varphi_{mpa} + A_{am} \varphi_{mpi} \big) \Big) \\
    &= \frac{4}{3} (\nabla_m \alpha) \varphi_{mia} - (\nabla_p X_q) \psi_{pqia} - (\nabla_p A_{mp}) \varphi_{mia} - \big( (\nabla_p A_{iq}) \varphi_{pqa} - (\nabla_p A_{aq}) \varphi_{pqi} \big) + \lot, \numberthis \label{eqn-chi-coclosed}
\end{align*}
where $\lot$ represents lower order terms. (That is, terms which involve no derivatives of $\alpha$, $X$, or $A$.)

Projecting the above onto its $7$-component (equivalently contracting with $\varphi_{iak}$), we obtain the relation
\begin{align*}
    0 &= 8 (\nabla_m \alpha) + 4 (\nabla_p X_q) \varphi_{pqm} - 4 (\nabla_p A_{pm}) + \lot,
\end{align*}
or equivalently
\begin{align}
\label{eqn-chi-coclosed-7}
    2 (\grad \alpha) + (\curl X) - (\div A) = \lot,
\end{align}
which we will make use of below to rewrite~\eqref{eqn-deturck-W} as~\eqref{eqn-deturck-W-2}.

Next we compute the first variations of $\tau_0$ and $\tau_3$ when the $G_2$-structure $\varphi$ has first variation~\eqref{eqn-flow3form}. For this we use~\cite[Proposition 3.10]{DGK25} which says that if $\del_t \varphi = Q \diamond \varphi$, then
$$ \del_t T_{pq} = \frac{1}{2} (\nabla_i Q_{pj} + \nabla_i Q_{jp} - \nabla_p Q_{ij}) \varphi_{ijq} + T_{pk} Q_{qk}. $$
Using this and~\eqref{eq:torsion-relations}, we can compute from $\tau_0 = \frac{4}{7} T_{kk}$ and $- (\tau_3')_{pq} = ( \frac{1}{2} (T_{pq} + T_{qp}) - \frac{1}{7} T_{kk} g_{pq})$ that
\begin{align*}
    \del_t \tau_0 &= \frac{4}{7} (\nabla_m X_m) + \lot, \\
    - \del_t (\tau_3')_{pq} &= \frac{1}{2} \big( \nabla_p X_q + \nabla_q X_p \big) + \frac{1}{2} \big( (\nabla_i A_{jp}) \varphi_{ijq} +(\nabla_i A_{jq}) \varphi_{ijp} \big) - \frac{1}{7} (\nabla_m X_m) g_{pq} + \lot.
\end{align*}
From the second equation above and $(\tau_3)_{ijk} = (\tau_3')_{ip} \varphi_{pjk} - (\tau_3')_{jp} \varphi_{pik} + (\tau_3')_{kp} \varphi_{pij}$, we then obtain
\begin{align*}
    \del_t (\tau_3)_{ijk} &= - \frac{1}{2} \bigg[ \Big( \nabla_i X_p + \nabla_p X_i - \frac{2}{7} (\nabla_m X_m) g_{ip} \Big) \varphi_{pjk} - \Big( \nabla_j X_p + \nabla_p X_j - \frac{2}{7} (\nabla_m X_m) g_{jp} \Big) \varphi_{pik} \\
    &\qquad \qquad + \Big( \nabla_k X_p + \nabla_p X_k - \frac{2}{7} (\nabla_m X_m) g_{kp} \Big) \varphi_{pij} \bigg] \\
    &\qquad - \frac{1}{2} \Big[ \big( (\nabla_p A_{iq}) \varphi_{pqm} + (\nabla_p A_{mq}) \varphi_{pqi} \big) \varphi_{mjk} - \big( (\nabla_p A_{jq}) \varphi_{pqm} + (\nabla_p A_{mq}) \varphi_{pqj} \big) \varphi_{mik} \\
    &\qquad \qquad \qquad + \big( (\nabla_p A_{kq}) \varphi_{pqm} + (\nabla_p A_{mq}) \varphi_{pqk} \big) \varphi_{mij} \Big] + \lot.
\end{align*}
Finally, from~\eqref{eqn-dilaton-aux} (recall that we have dropped the tildes for the auxiliary structures in this section), or immediately from~\eqref{eq:f-evol} (with $A$ replaced by $Q = \frac{1}{3} \alpha g - \frac{1}{3} X \hook \varphi + A$) we obtain
\begin{align*}
    \del_t f = \frac{7}{12} \alpha.
\end{align*}

Collecting these computations, from~\eqref{eq:P-op} we find that the linearization of $P$ at $\varphi$ is
\begin{align*}
    &(DP)_{\varphi} (\chi) \\
    &= e^{8f} d \bigg( {\Big( C - \frac{10}{7} \Big) (\nabla_m X_m) \varphi_{ijk}} {+ \frac{7}{3} (\nabla_m \alpha) \psi_{mijk}} \\
    &\qquad \qquad {- \frac{1}{2} \bigg[ \Big( \nabla_i X_p + \nabla_p X_i - \frac{2}{7} (\nabla_m X_m) g_{ip} \Big) \varphi_{pjk} - \Big( \nabla_j X_p + \nabla_p X_j - \frac{2}{7} (\nabla_m X_m) g_{jp} \Big) \varphi_{pik}} \\
    &\qquad \qquad \qquad {+ \Big( \nabla_k X_p + \nabla_p X_k - \frac{2}{7} (\nabla_m X_m) g_{kp} \Big) \varphi_{pij} \bigg]} \\
    &\qquad \qquad {- \frac{1}{2} \Big[ \big( (\nabla_p A_{iq}) \varphi_{pqm} + (\nabla_p A_{mq}) \varphi_{pqi} \big) \varphi_{mjk} - \big( (\nabla_p A_{jq}) \varphi_{pqm} + (\nabla_p A_{mq}) \varphi_{pqj} \big) \varphi_{mik}} \\
    &\qquad \qquad \qquad {+ \big( (\nabla_p A_{kq}) \varphi_{pqm} + (\nabla_p A_{mq}) \varphi_{pqk} \big) \varphi_{mij} \Big]} + \lot \bigg). \numberthis \label{eq:DP}
\end{align*}

We compare this against $-\Delta_d \chi = -dd^* \chi$. Firstly, from~\eqref{eqn-chi} we see that
\begin{align*}
    (-d^* \chi)_{ijk} &= \nabla_p \Big( \frac{4}{3} \alpha \psi_{pijk} - \big( X_p \varphi_{ijk} - X_i \varphi_{pjk} + X_j \varphi_{pik} - X_k \varphi_{pij} \big) \\
    &\qquad \qquad + \big( A_{pm} \psi_{mijk} - A_{im} \psi_{mpjk} + A_{jm} \psi_{mpik} - A_{km} \psi_{mpij} \big) \Big) \\
    &= {\frac{4}{3} (\nabla_p \alpha) \psi_{pijk}} {- (\nabla_m X_m) \varphi_{ijk}} + \big( (\nabla_p X_i) \varphi_{pjk} - (\nabla_p X_j) \varphi_{pik} + (\nabla_p X_k) \varphi_{pij} \big) \\
    &\qquad \qquad {+ (\nabla_p A_{pm}) \psi_{mijk}} + \big( (\nabla_p A_{qi}) \psi_{pqjk} - (\nabla_p A_{qj}) \psi_{pqik} + (\nabla_p A_{qk}) \psi_{pqij} \big) + \lot. \numberthis \label{eq:d-star-chi}
\end{align*}

By decomposing the $2$-tensor $\nabla_p X_q$ into its components, we also see that
\begin{align*}
    &\qquad (\nabla_p X_i) \varphi_{pjk} - (\nabla_p X_j) \varphi_{pik} + (\nabla_p X_k) \varphi_{pij} \\
    &= {\frac{3}{7} (\nabla_m X_m) \varphi_{ijk}} {{}- \frac{1}{2} \Big( (\nabla_i X_p - \nabla_p X_i) \varphi_{pjk} - (\nabla_j X_p - \nabla_p X_j) \varphi_{pik} + (\nabla_k X_p - \nabla_p X_k) \varphi_{pij} \Big)} \\
    &\qquad {+ \frac{1}{2} \bigg( \Big( \nabla_i X_p + \nabla_p X_i - \frac{2}{7} (\nabla_m X_m) g_{ip} \Big) \varphi_{pjk} - \Big( \nabla_j X_p + \nabla_p X_j - \frac{2}{7} (\nabla_m X_m) g_{jp} \Big) \varphi_{pik}} \\
    &\qquad \qquad {+ \Big( \nabla_k X_p + \nabla_p X_k - \frac{2}{7} (\nabla_m X_m) g_{kp} \Big) \varphi_{pij} \bigg)}. \numberthis \label{eq:d-star-chi-2}
\end{align*}

To simplify the $\nabla A$ terms in~\eqref{eq:d-star-chi}, we exploit the fact that $d \chi = 0$. Explicitly, using~\eqref{eqn-chi-coclosed} we have
\begin{align*}
    &\qquad \big( (\nabla_p A_{qi}) \varphi_{pqm} + (\nabla_p A_{qm}) \varphi_{pqi} \big) \varphi_{mjk} \\
    &= 2 (\nabla_p A_{qi}) \varphi_{pqm} \varphi_{mjk} - \frac{4}{3} (\nabla_p \alpha) \varphi_{pim} \varphi_{mjk} + (\nabla_p X_q) \psi_{pqim} \varphi_{mjk} + (\nabla_p A_{pq}) \varphi_{qim} \varphi_{mjk} + \lot \\
    &= 2 (\nabla_j A_{ki}) - 2 (\nabla_k A_{ji}) - 2 (\nabla_p A_{qi}) \psi_{pqjk} - \frac{4}{3} (\nabla_j \alpha) g_{ik} + \frac{4}{3} (\nabla_k \alpha) g_{ij} + \frac{4}{3} (\nabla_m \alpha) \psi_{mijk} \\
    &\qquad + (\nabla_j X_p - \nabla_p X_j) \varphi_{pik} - (\nabla_k X_p - \nabla_p X_k) \varphi_{pij} + (\nabla_p X_q) \varphi_{pqk} g_{ij} - (\nabla_p X_q) \varphi_{pqj} g_{ik} \\
    &\qquad + (\nabla_p A_{pj}) g_{ik} - (\nabla_p A_{pk}) g_{ij} - (\nabla_p A_{pm}) \psi_{mijk} + \lot.
\end{align*}
Skew-symmetrizing the above in the indices $i,j,k$ and rearranging yields
\begin{align*}
    & \qquad \big( (\nabla_p A_{qi}) \psi_{pqjk} - (\nabla_p A_{qj}) \psi_{pqik} + (\nabla_p A_{qk}) \psi_{pqij} \big) \\
    &= {2 (\nabla_m \alpha) \psi_{mijk} - \frac{3}{2} (\nabla_p A_{pm}) \psi_{mijk} - \big( (\nabla_i X_p - \nabla_p X_i) \varphi_{pjk}} \\
    &\qquad {- (\nabla_j X_p - \nabla_p X_j) \varphi_{pik} + (\nabla_k X_p - \nabla_p X_k) \varphi_{pij} \big)} \\
    &\qquad {- \frac{1}{2} \Big( \big( (\nabla_p A_{iq}) \varphi_{pqm} + (\nabla_p A_{mq}) \varphi_{pqi} \big) \varphi_{mjk} - \big( (\nabla_p A_{jq}) \varphi_{pqm} + (\nabla_p A_{mq}) \varphi_{pqj} \big) \varphi_{mik}} \\
    &\qquad \qquad {+ \big( (\nabla_p A_{kq}) \varphi_{pqm} + (\nabla_p A_{mq}) \varphi_{pqk} \big) \varphi_{mij} \Big)} + \lot. \numberthis \label{eq:d-star-chi-3}
\end{align*}

Substituting~\eqref{eq:d-star-chi-2} and~\eqref{eq:d-star-chi-3} into~\eqref{eq:d-star-chi} gives
\begin{align*}
    (-d^* \chi)_{ijk} &= {- \frac{4}{7} (\nabla_m X_m) \varphi_{ijk}} {+ \frac{10}{3} (\nabla_m \alpha) \psi_{mijk} - \frac{1}{2} (\nabla_p A_{pm}) \psi_{mijk}} \\
    &\qquad {- \frac{3}{2} \big( (\nabla_i X_p - \nabla_p X_i) \varphi_{pjk} - (\nabla_j X_p - \nabla_p X_j) \varphi_{pik} + (\nabla_k X_p - \nabla_p X_k) \varphi_{pij} \big)} \\
    &\qquad {+ \frac{1}{2} \bigg( \Big( \nabla_i X_p + \nabla_p X_i - \frac{2}{7} (\nabla_m X_m) g_{ip} \Big) \varphi_{pjk} - \Big( \nabla_j X_p + \nabla_p X_j - \frac{2}{7} (\nabla_m X_m) g_{jp} \Big) \varphi_{pik}} \\
    &\qquad \qquad {+ \Big( \nabla_k X_p + \nabla_p X_k - \frac{2}{7} (\nabla_m X_m) g_{kp} \Big) \varphi_{pij} \bigg)} \\
    &\qquad {- \frac{1}{2} \Big( \big( (\nabla_p A_{iq}) \varphi_{pqm} + (\nabla_p A_{mq}) \varphi_{pqi} \big) \varphi_{mjk} - \big( (\nabla_p A_{jq}) \varphi_{pqm} + (\nabla_p A_{mq}) \varphi_{pqj} \big) \varphi_{mik}} \\
    &\qquad \qquad {+ \big( (\nabla_p A_{kq}) \varphi_{pqm} + (\nabla_p A_{mq}) \varphi_{pqk} \big) \varphi_{mij} \Big)} + \lot.
\end{align*}

Hence, from the above and~\eqref{eq:DP} we obtain
\begin{align*}
    & \qquad (DP)_{\varphi} (\chi) + e^{8f} \Delta_d \chi \\
    &= e^{8f} d \bigg( {\Big( C - \frac{6}{7} \Big) (\nabla_m X_m) \varphi_{ijk}} {- (\nabla_m \alpha) \psi_{mijk} + \frac{1}{2} (\nabla_p A_{pm}) \psi_{mijk}} \\
    &\qquad \qquad {+ \frac{3}{2} \big( (\nabla_i X_p - \nabla_p X_i) \varphi_{pjk} - (\nabla_j X_p - \nabla_p X_j) \varphi_{pik} + (\nabla_k X_p - \nabla_p X_k) \varphi_{pij} \big)} \\
    &\qquad \qquad {- \bigg[ \Big( \nabla_i X_p + \nabla_p X_i - \frac{2}{7} (\nabla_m X_m) g_{ip} \Big) \varphi_{pjk} - \Big( \nabla_j X_p + \nabla_p X_j - \frac{2}{7} (\nabla_m X_m) g_{jp} \Big) \varphi_{pik}} \\
    &\qquad \qquad \qquad {+ \Big( \nabla_k X_p + \nabla_p X_k - \frac{2}{7} (\nabla_m X_m) g_{kp} \Big) \varphi_{pij} \bigg]} + \lot \bigg). \numberthis \label{eq:DPplusLaplace}
\end{align*}

We can significantly simplify the above as follows. First, observe that
\begin{align*}
\big( d( X \hook \varphi) \big){}_{ijk} & = \nabla_i (X_p \varphi_{pjk}) - \nabla_j (X_p \varphi_{pik}) + \nabla_k (X_p \varphi_{pij}) \\
& =  (\nabla_i X_p) \varphi_{pjk} - (\nabla_j X_p) \varphi_{pik} + (\nabla_k X_p) \varphi_{pij} \big) + \lot.
\end{align*}
Taking $d$ of the above and using~\eqref{eq:d-star-chi-2}, we have
\begin{align*}
    & \qquad d \big( (\nabla_i X_p) \varphi_{pjk} - (\nabla_j X_p) \varphi_{pik} + (\nabla_k X_p) \varphi_{pij} \big) \\
    &= d \bigg( {\frac{3}{7} (\nabla_m X_m) \varphi_{ijk}} {+ \frac{1}{2} \big( (\nabla_i X_p - \nabla_p X_i) \varphi_{pjk} - (\nabla_j X_p - \nabla_p X_j) \varphi_{pik} + (\nabla_k X_p - \nabla_p X_k) \varphi_{pij} \big)} \\
    &\qquad {+ \frac{1}{2} \bigg( \Big( \nabla_i X_p + \nabla_p X_i - \frac{2}{7} (\nabla_m X_m) g_{ip} \Big) \varphi_{pjk} - \Big( \nabla_j X_p + \nabla_p X_j - \frac{2}{7} (\nabla_m X_m) g_{jp} \Big) \varphi_{pik}} \\
    &\qquad \qquad {+ \Big( \nabla_k X_p + \nabla_p X_k - \frac{2}{7} (\nabla_m X_m) g_{kp} \Big) \varphi_{pij} \bigg)} + \lot \bigg).
\end{align*}
Adding twice the above to the expression~\eqref{eq:DPplusLaplace} yields
\begin{align*}
    & \qquad (DP)_{\varphi} (\chi) + e^{8f} \Delta_d \chi \\
    &= e^{8f} d \Big( {C (\nabla_m X_m) \varphi_{ijk}} {- (\nabla_m \alpha) \psi_{mijk} + \frac{1}{2} (\nabla_p A_{pm}) \psi_{mijk}} \\
    &\qquad \qquad {+ \frac{5}{2} \big( (\nabla_i X_p - \nabla_p X_i) \varphi_{pjk} - (\nabla_j X_p - \nabla_p X_j) \varphi_{pik} + (\nabla_k X_p - \nabla_p X_k) \varphi_{pij} \big)} + \lot \Big). \numberthis \label{eq:DPplusLaplace-2}
\end{align*}

We can rewrite the above in terms of $\curl X$, $\grad \alpha$, and $\div A$ as follows. From $(\curl X)_m = (\nabla_a X_b) \varphi_{abm}$ we get
\begin{align*}
    \big( (\curl X) \hook \psi \big){}_{ijk} &= {- \big( (\nabla_i X_p - \nabla_p X_i) \varphi_{pjk} - (\nabla_j X_p - \nabla_p X_j) \varphi_{pik} + (\nabla_k X_p - \nabla_p X_k) \varphi_{pij} \big)}.
\end{align*}
We also have
\begin{align*}
    \big( (\grad \alpha) \hook \psi \big){}_{ijk} &= {(\nabla_m \alpha) \psi_{mijk}}, \\
    \big( (\div A) \hook \psi \big){}_{ijk} &= {(\nabla_p A_{pm}) \psi_{mijk}}.
\end{align*}
Defining a vector field
\begin{align} \label{eqn-deturck-W}
    W(\chi) := 2 (\curl X).
\end{align}
Using~\eqref{eqn-chi-coclosed-7}, we have
\begin{align} \label{eqn-deturck-W-2}
    W(\chi) = 2 (\curl X) = (\grad \alpha) + \frac{5}{2} (\curl X) - \frac{1}{2} (\div A) + \lot.
\end{align}
Thus, using~\eqref{eqn-deturck-W-2} and the fact that $d \psi = 0$, equation~\eqref{eq:DPplusLaplace-2} finally becomes
\begin{align} \label{eq:DP-final}
    (DP)_{\varphi} (\chi) = e^{8f} \Big( - \cal{L}_{W(\chi)} \psi - \Delta_d \chi + C d \big( {(\div X) \varphi} \big) + d F(\chi) \Big),
\end{align}
for some zeroth-order operator $F$.

In order to be able to apply DeTurck's Trick, we need to verify that the vector field $W(\chi)$ from~\eqref{eqn-deturck-W} is compatible, since the operator $\varphi \mapsto P(\varphi)$ depends on a reference volume form $\vol_R$ and thus is invariant only under volume-preserving diffeomorphisms. But since $\varphi$ is assumed to be coclosed, we have~\cite[proof of Proposition 4.4]{Kar10} that
\begin{align*}
    \div (\curl X) = 0
\end{align*}
and so the gauge-fixing vector field $W(\chi)$ is divergence-free. Thus the $1$-parameter family of diffeomorphisms it generates is volume-preserving, as required.

\subsection{Positivity of the principal symbol}
\label{subsect-pos-symbol}

We now compute the symbol of the operator $P$ at $\varphi$. For simplicity, we consider a slightly modified operator $P'$ with linearization
\begin{align*}
    (DP')_{\varphi} (\chi) = - \Delta_d \chi + C d \big( (\div X) \varphi \big).
\end{align*}
Up to an overall conformal factor (which does not affect the positivity of the principal symbol), this operator $P'$ is the one we obtain from $P$ after applying DeTurck's Trick to remove the Lie derivative term.

Note that since from~\eqref{eqn-chi} we have
\begin{align} \label{eq:chi-orthog}
    \chi = \frac{4}{3} \alpha - X^\flat \wedge \varphi + A \diamond \psi,
\end{align}
it follows by contraction of the above with $\varphi$ on three indices that the vector field $X$ is given by
\begin{align}
\label{eqn-X}
    X_m = -\frac{1}{24} \chi_{mpqr} \varphi_{pqr}.
\end{align}
As such,
\begin{align*}
    \div X = - \frac{1}{24} (\nabla_m \chi_{mpqr}) \varphi_{pqr} + \lot
\end{align*}
and so
\begin{align*}
    \big( d \big( (\div X) \varphi \big) \big){}_{ijkl} &= - \frac{1}{24} \big( (\nabla_i \nabla_m \chi_{mpqr}) \varphi_{pqr} \varphi_{jkl} - (\nabla_j \nabla_m \chi_{mpqr}) \varphi_{pqr} \varphi_{ikl} \\
    &\qquad \qquad + (\nabla_k \nabla_m \chi_{mpqr}) \varphi_{pqr} \varphi_{ijl} - (\nabla_l \nabla_m \chi_{mpqr}) \varphi_{pqr} \varphi_{ijk} \big) + \lot.
\end{align*}
From this, we see that the principal symbol of this component is
\begin{align*}
    \sigma_\xi \big[ d \big( (\div X) \varphi \big) \big]{}_{ijkl} &= - \frac{1}{24} \big( \xi_i \xi_m \chi_{mpqr} \varphi_{pqr} \varphi_{jkl} - \xi_j \xi_m \chi_{mpqr} \varphi_{pqr} \varphi_{ikl} \\
    &\qquad \qquad + \xi_k \xi_m \chi_{mpqr} \varphi_{pqr} \varphi_{ijl} - \xi_l \xi_m \chi_{mpqr} \varphi_{pqr} \varphi_{ijk} \big).
\end{align*}
Recall also that the symbol of the Hodge Laplacian term is
\begin{align*}
    \sigma_\xi \big[ -\Delta_d \chi \big]{}_{ijkl} = \xi_p \xi_p \chi_{ijkl}.
\end{align*}

Denoting by $\dinner{\cdot}{\cdot}$ and $\|\cdot\|$ the tensor inner product and tensor norm, we thus compute
\begin{align*}
    \dbinner{\sigma_\xi \big[ (DP')_{\varphi}(\chi) \big]}{\chi} &= \xi_p \xi_p \chi_{ijkl} \chi_{ijkl} \\
    &\qquad - \frac{1}{24} C \big( \xi_i \xi_m \chi_{mpqr} \varphi_{pqr} \varphi_{jkl} - \xi_j \xi_m \chi_{mpqr} \varphi_{pqr} \varphi_{ikl} \\
    &\qquad \qquad \qquad + \xi_k \xi_m \chi_{mpqr} \varphi_{pqr} \varphi_{ijl} - \xi_l \xi_m \chi_{mpqr} \varphi_{pqr} \varphi_{ijk} \big) \chi_{ijkl} \\
    &= \xi_p \xi_p \chi_{ijkl} \chi_{ijkl} \\
    &\qquad - \frac{1}{24} C \xi_m \chi_{mpqr} \varphi_{pqr} \big( \xi_i \varphi_{jkl} - \xi_j \varphi_{ikl} + \xi_k \varphi_{ijl} - \xi_l \varphi_{ijk} \big) \chi_{ijkl} \\
    &= \|\xi\|^2 \|\chi\|^2 - \frac{1}{24} C \dinner{\xi \hook \chi}{\varphi} \dinner{\chi}{\xi^\flat \wedge \varphi} \\
    &= \|\xi\|^2 \|\chi\|^2 - \frac{1}{6} C \dinner{\xi \hook \chi}{\varphi}^2. \numberthis \label{eq:symbol-pos}
\end{align*}
Using~\eqref{eqn-X} we have
\begin{align*}
    \dinner{\xi \hook \chi}{\varphi} &= \xi_m \chi_{mpqr} \varphi_{pqr} = - 24 \xi_m X_m = -24 \dinner{\xi}{X}.
\end{align*}
Further, we can check that
\begin{align*}
    \|-X^\flat \wedge \varphi\|^2 &= \big( X_i \varphi_{jkl} - X_j \varphi_{ikl} + X_k \varphi_{ijl} - X_l \varphi_{ijk} \big) \big( X_i \varphi_{jkl} - X_j \varphi_{ikl} + X_k \varphi_{ijl} - X_l \varphi_{ijk} \big) \\
    &= 4 X_i \varphi_{jkl} \big( X_i \varphi_{jkl} - X_j \varphi_{ikl} + X_k \varphi_{ijl} - X_l \varphi_{ijk} \big) \\
    &= 168 \|X\|^2 - 12 X_i \varphi_{jkl} X_j \varphi_{ikl} \\ 
    &= 96 \|X\|^2.
\end{align*}
Combining the above two computations and recalling the orthogonal decomposition~\eqref{eq:chi-orthog}, we have
\begin{align*}
    \dinner{\xi \hook \chi}{\varphi}^2 = 576 \dinner{\xi}{X}^2 \leq 576 \|\xi\|^2 \|X\|^2 = 6 \|\xi\|^2 \|-X^\flat \wedge \varphi\|^2 \leq 6 \|\xi\|^2 \|\chi\|^2,
\end{align*}
so
$$ - \frac{1}{6} \dinner{\xi \hook \chi}{\varphi}^2 \geq - \|\xi\|^2 \|\chi\|^2. $$
In particular, if $C \geq 0$, then~\eqref{eq:symbol-pos} becomes
\begin{align*}
    \dbinner{\sigma_\xi \big[ (DP')_{\varphi}(\chi) \big]}{\chi} &= \|\xi\|^2 \|\chi\|^2 - \frac{1}{6} C \dinner{\xi \hook \chi}{\varphi}^2 \geq (1 - C) \|\xi\|^2 \|\chi\|^2,
\end{align*}
while if $C < 0$ then~\eqref{eq:symbol-pos} trivially gives
$$ \dbinner{\sigma_\xi \big[ (DP')_{\varphi}(\chi) \big]}{\chi} \geq \|\xi\|^2 \|\chi\|^2. $$
As such, we see that as long as $C < 1$, then the principal symbol of the operator $P'$ (and hence also of $P$) is strictly positive.

Using the method of Bryant--Xu~\cite{BX11} and Grigorian~\cite{Gri13}, the above observation then yields the following result, which was stated as Theorem~\ref{thm:STE-intro} in the introduction. (Note here that the $G_2$-structures below denote the original conformally coclosed $G_2$-structures, and not the auxiliary coclosed $G_2$-structures.)

\begin{thm}
\label{thm-G2-Anomaly-flow-ste}
    If $C< 1$, then the modified $G_2$-anomaly flow
    \begin{align}
        \del_t (e^{-2f} \psi) = -dH_C = d \big( e^{2f} d^* (e^{-2f} \psi) \big) + (C-2) d \big( (\tr T) \varphi \big). \tag{\ref{eqn-G2-Anomaly-flow}}
    \end{align}
    has short-time existence and uniqueness under the restriction to conformally coclosed $G_2$-structures.
\end{thm}

\begin{rmk}
    This result generalizes (and recovers) the result of Grigorian from~\cite{Gri16}, where it is stated that the modified Laplacian coflow
    \begin{align*}
        \del_t \psi = d d^* \psi + kd \big( (A - \tr T) \varphi \big)
    \end{align*}
    has short-time existence and uniqueness if $k > 1$ under the restriction to coclosed $G_2$-structures. (Grigorian's modified coflow corresponds to the modified $G_2$-anomaly flow~\eqref{eqn-G2-Anomaly-flow} when $f$ is constant, $C = - k$, and $k A = - 2$, so the condition $k > 1$  corresponds to $C < 1$ in our notation).
\end{rmk}

\begin{rmk}
It is natural to ask the following question. Since these flows have short-time existence and uniqueness for $C<1$, can we find a particular $C_0 < 1$ for which it is true that if $\tau_0$ is initially zero, then it stays zero along the flow (using a maximum principle argument, for example). If this were the case, then the flow for $C = \frac{4}{3}$ would coincide with the flow for $C_0$ for such initial data, since we would have $\tr T = 0$ for all time. In such case, it would follow from~\eqref{eqn-fixed-pt-G2-Anomaly-flow-1} and its following paragraph that the fixed points would be torsion-free. The authors attempted to find such a $C_0$, but were unsuccessful. It would be interesting to find an explicit counterexample, perhaps on homogeneous spaces.
\end{rmk}

\section{Summary}
\label{sect-summary}

We summarize the results of this paper in Table~\ref{table2} below.

\begin{table}[H]
{\smaller[1]
\begin{center}
\begin{tabular}{| m{1.5cm} | m{6.4cm} | m{6.8cm} |}
    \hline
    & \textbf{Modified Laplacian Coflow} & \textbf{Modified $G_2$-Anomaly Flow} \\
    \hline
    Evolution Equation & $\del_t \psi = d d^* \psi + (C-2) d \big( (\tr T) \varphi \big)$ & $\del_t (e^{-2f} \psi) = d \big(e^{2f} d^* (e^{-2f} \psi) \big) + (C-2) d \big( (\tr T) \varphi \big)$ \\
    \hline
    Preserved Condition & $d \psi = 0$ & $d (e^{-2f} \psi) = 0$ \\
    \hline
    Short-Time Existence/ Uniqueness & If $C < 1$ (sufficient, may not be necessary) & If $C < 1$ (sufficient, may not be necessary) \\
    \hline
    \multirow{4}{1.5cm}{Fixed Points} & If $C \neq 1$, then $\tau_0$ constant, $\div T = 0$, and $\tau_3$ has constant norm and is coclosed. & If $C \neq 1$, then $\tau_0 = K \exp{(- \frac{3 (C - \frac{4}{3})}{2 (C-1)}f)}$ for some constant $K$. \\
    & & If $C = \frac{4}{3}$, then $\tau_0$ is constant and $(g,H,f)$ is a generalized Ricci soliton. \\
    & If $C = \frac{10}{7}$, then $\varphi$ is torsion-free or nearly parallel. & If $C = \frac{10}{7}$, then $f$ is constant and $\varphi$ is torsion-free or nearly parallel. \\
    & If $C > \frac{10}{7}$, then $\varphi$ is torsion-free. & If $C > \frac{10}{7}$, then $f$ is constant and $\varphi$ is torsion-free. \\
    \hline
\end{tabular}
\end{center}
}
\caption{Summary of fixed points and short-time existence results.} \label{table2}
\end{table}

Interestingly, the situation is reminiscent of the Heisenberg Uncertainty Principle, in that for these flows, we seem to be able to either know short-time existence and uniqueness, or to know properties of the fixed points, but not both simultaneously (other than those properties corresponding to $C \neq 1$). In particular, the case of interest in physics, which corresponds to $C = \frac{4}{3}$, is outside the current range of values of $C < 1$ for which we can establish short-time existence and uniqueness using known methods.

\appendix

\section{Decomposition of exact 4-forms}

\label{app-decomp-exact-4-forms}

In this appendix, we obtain expressions for the $1$-, $7$-, and $27$-component of an exact $4$-form on a $7$-manifold $M$ with $G_2$-structure $\varphi$. These are derived using methods from~\cite{DGK25,Kar09}. We sketch the argument and leave some of the details to the reader. Note that we make no assumptions here on the torsion, so this result could find wider application.

Consider an exact $4$-form
\begin{align*}
    \eta = d (B \diamond \varphi)
\end{align*}
for some tensor $B$. We want to write
\begin{align*}
    \eta = S \diamond \psi
\end{align*}
for some tensor $S$.

To do this, we first compute the tensor $\eta^\psi_{ia} = \eta_{ijkl} \psi_{ajkl}$. We have
\begin{align*}
    \eta^\psi_{ia} &= \eta_{ijkl} \psi_{ajkl} \\
    &= \big( \nabla_i (B \diamond \varphi)_{jkl} - \nabla_j (B \diamond \varphi)_{ikl} + \nabla_k (B \diamond \varphi)_{ijl} - \nabla_l (B \diamond \varphi)_{ijk} \big) \psi_{ajkl} \\
    &= \nabla_i (B \diamond \varphi)_{jkl} \psi_{ajkl} - 3 \nabla_j (B \diamond \varphi)_{ikl} \psi_{ajkl} \\
    &= \nabla_i \big( B_{jp} \varphi_{pkl} - B_{kp} \varphi_{pjl} + B_{lp} \varphi_{pjk} \big) \psi_{ajkl} - 3 \nabla_j \big( B_{ip} \varphi_{pkl} - B_{kp} \varphi_{pil} + B_{lp} \varphi_{pik} \big) \psi_{ajkl} \\
    &= 3 (\nabla_i B_{jp}) \varphi_{pkl} \psi_{ajkl} - 3 (\nabla_j B_{ip}) \varphi_{pkl} \psi_{ajkl} + 6 (\nabla_j B_{kp}) \varphi_{pil} \psi_{ajkl} \\
    &\qquad + 3 B_{jp} T_{im} \psi_{mpkl} \psi_{ajkl} - 3 B_{ip} T_{jm} \psi_{mpkl} \psi_{ajkl} + 6 B_{kp} T_{jm} \psi_{mpil} \psi_{ajkl}.
\end{align*}    
Using contraction identities for $\varphi$ and $\psi$, this becomes
\begin{align*}
    \eta^\psi_{ia} &= - 12 (\nabla_i B_{jp}) \varphi_{paj} + 12 (\nabla_j B_{ip}) \varphi_{paj} \\
    &\qquad + 6 (\nabla_j B_{kp}) \big( g_{pa} \varphi_{ijk} - g_{pj} \varphi_{iak} + g_{pk} \varphi_{iaj} - g_{ia} \varphi_{pjk} + g_{ij} \varphi_{pak} - g_{ik} \varphi_{paj} \big) \\
    &\qquad + 3 B_{jp} T_{im} \big( 4 g_{ma} g_{pj} - 4 g_{mj} g_{pa} - 2 \psi_{mpaj} \big) - 3 B_{ip} T_{jm} \big( 4 g_{ma} g_{pj} - 4 g_{mj} g_{pa} - 2 \psi_{mpaj} \big) \\
    &\qquad + 6 B_{kp} T_{jm} \big( - \varphi_{api} \varphi_{mjk} - \varphi_{mai} \varphi_{pjk} - \varphi_{mpa} \varphi_{ijk} \\
    &\qquad \qquad \qquad \qquad + g_{ma} g_{pj} g_{ik} + g_{mj} g_{pk} g_{ia} + g_{mk} g_{pa} g_{ij} - g_{ma} g_{pk} g_{ij} - g_{mj} g_{pa} g_{ik} - g_{mk} g_{pj} g_{ia} \\
    &\qquad \qquad \qquad \qquad - g_{ma} \psi_{pijk} - g_{pa} \psi_{imjk} - g_{ia} \psi_{mpjk} + g_{aj} \psi_{mpik} - g_{ak} \psi_{mpij} \big)
\end{align*}
which then further simplifies to 
\begin{align*}
    \eta^\psi_{ia} &= - 6 (\nabla_i B_{pq}) \varphi_{pqa} + 6 (\nabla_p B_{iq}) \varphi_{pqa} + 6 (\nabla_p B_{qa}) \varphi_{pqi} \\
    &\qquad - 6 (\nabla_m B_{km}) \varphi_{kia} + 6 (\nabla_k B_{mm}) \varphi_{kia} - 6 (\nabla_m B_{pq}) \varphi_{mpq} g_{ia} \\
    &\qquad + 6 B_{mm} T_{ia} - 6 B_{im} T_{ma} - 6 B_{pq} T_{qp} g_{ia} + 6 T_{mm} B_{ia} - 6 T_{im} B_{ma} + 6 B_{mm} T_{kk} g_{ia} \\
    &\qquad + 6 B_{pq} \varphi_{pqm} T_{mk} \varphi_{kia} + 6 T_{pq} \varphi_{pqm} B_{mk} \varphi_{kia} - 6 B_{pq} T_{rs} \varphi_{pri} \varphi_{qsa} \\
    &\qquad - 6 B_{im} T_{pq} \psi_{pqma} + 6 B_{pq} \psi_{pqim} T_{ma} - 6 T_{im} B_{pq} \psi_{pqma} + 6 T_{pq} \psi_{pqim} B_{ma} \\
    &\qquad - 6 B_{pq} \psi_{pqim} T_{am} - 6 T_{pq} \psi_{pqim} B_{am} + 6 B_{pq} T_{rs} \psi_{pqrs} g_{ia}.
\end{align*}

Using~\cite[Corollary 2.33]{DGK25}, we have that $\eta = S \diamond \psi$ where
\begin{align}
    \tr S = \frac{1}{96} \tr \eta^\psi &= - \frac{1}{2} (\nabla_m B_{pq}) \varphi_{mpq} + \frac{1}{2} B_{mm} T_{kk} - \frac{1}{2} B_{pq} T_{qp} + \frac{1}{4} B_{pq} T_{rs} \psi_{pqrs}, \numberthis \label{eqn-decomp-exact-4-form-1}
\end{align}
\begin{align*}
    (S_7)_{ia} = \frac{1}{36} (\eta^\psi_7)_{ia} &= - \frac{1}{6} (\nabla_m B_{km}) \varphi_{kia} + \frac{1}{6} (\nabla_k B_{mm}) \varphi_{kia} - \frac{1}{12} (\nabla_m B_{pq}) \psi_{mpqk} \varphi_{kia} \\
    &\qquad + \frac{1}{12} B_{mm} T_{pq} \varphi_{pqk} \varphi_{kia} + \frac{1}{12} T_{mm} B_{pq} \varphi_{pqk} \varphi_{kia} \\
    &\qquad - \frac{1}{12} B_{pm} T_{mq} \varphi_{pqk} \varphi_{kia} - \frac{1}{12} T_{pm} B_{mq} \varphi_{pqk} \varphi_{kia} \\
    &\qquad + \frac{1}{12} B_{pq} \varphi_{pqm} T_{km} \varphi_{kia} + \frac{1}{12} T_{pq} \varphi_{pqm} B_{km} \varphi_{kia}, \numberthis \label{eqn-decomp-exact-4-form-7}
\end{align*}
and
\begin{align*}
    (S_{27})_{ia} = \frac{1}{12} (\eta^\psi_{27})_{ia} &= - \frac{1}{4} \big( (\nabla_i B_{pq}) \varphi_{pqa} + (\nabla_a B_{pq}) \varphi_{pqi} \big) + \frac{1}{4} \big( (\nabla_p B_{iq}) \varphi_{pqa} + (\nabla_p B_{aq}) \varphi_{pqi} \big) \\
    &\qquad + \frac{1}{4} \big( (\nabla_p B_{qi}) \varphi_{pqa} + (\nabla_p B_{qa}) \varphi_{pqi} \big) + \frac{1}{14} (\nabla_m B_{pq}) \varphi_{mpq} g_{ia} \\
    &\qquad + \frac{1}{4} B_{mm} \big( T_{ia} + T_{ai} \big) - \frac{1}{4} \big( B_{im} T_{ma} + B_{am} T_{mi} \big) \\
    &\qquad + \frac{1}{4} T_{mm} \big( B_{ia} + B_{ai} \big) - \frac{1}{4} \big( T_{im} B_{ma} + T_{am} B_{mi} \big) \\
    &\qquad + \frac{1}{4} B_{pq} \big( \psi_{pqim} T_{ma} + \psi_{pqam} T_{mi} \big) + \frac{1}{4} T_{pq} \big( \psi_{pqim} B_{ma} + \psi_{pqam} B_{mi} \big) \\
    &\qquad - \frac{1}{4} B_{pq} T_{rs} \big( \varphi_{pri} \varphi_{qsa} + \varphi_{pra} \varphi_{qsi} \big) \\
    &\qquad - \frac{1}{14} B_{mm} T_{kk} g_{ia} + \frac{1}{14} B_{pq} T_{qp} g_{ia} + \frac{3}{14} B_{pq} T_{rs} \psi_{pqrs} g_{ia}. \label{eqn-decomp-exact-4-form-27} \numberthis
\end{align*}

\bibliography{biblio}

\bibliographystyle{siam}

\end{document}